\documentclass[10pt,lettersize,reqno,fullpage]{article}
\usepackage{amsmath, amssymb, amsfonts}
\usepackage{bbm, mathtools, extarrows, url, amsthm}
\usepackage[all,cmtip]{xy} 
\xyoption{web}

\newtheoremstyle{mystyle} 
  {}
  {}
  {\itshape}
  {}
  {\bfseries}
  {.}
  { }
  {}
\theoremstyle{mystyle}
\newtheorem{thm}{Theorem}[section]
\newtheorem{cor}[thm]{Corollary}
\newtheorem{lmm}[thm]{Lemma}
\newtheorem{prpn}[thm]{Proposition}
\newtheorem{rem}[thm]{Remark}

\newtheorem*{thma}{Theorem A}
\newtheorem*{thmb}{Theorem B}
\newtheorem*{thmc}{Theorem C}
\newtheorem*{obsb}{Observation b}

\theoremstyle{definition}
\newtheorem{defn}[thm]{Definition}

\numberwithin{equation}{section}

\DeclareMathAlphabet{\mathpzc}{OT1}{pzc}{m}{it}

\newcommand{\spn}{\mathpzc{sp}}
\newcommand{\wt}{\widetilde}
\newcommand{\w}{\textup{w}}

\newcommand{\R}{\mathbb{R}}
\newcommand{\Z}{\mathbb{Z}}

\newcommand{\hf}{\hspace*{0.5cm}}
\newcommand{\bgd}{\begin{displaymath}}
\newcommand{\edd}{\end{displaymath}}
\newcommand{\bge}{\begin{equation}}
\newcommand{\ede}{\end{equation}}
\newcommand{\bgc}{\begin{center}}
\newcommand{\edc}{\end{center}}

\title{On the cohomology ring and upper characteristic rank of Grassmannian of oriented $3$-planes}
\author{Somnath Basu \and Prateep Chakraborty \thanks{The second author is funded by an NBHM post-doctoral fellowship.}}

\begin{document}

\maketitle


\begin{abstract}
In this paper we study the mod $2$ cohomology ring of the Grasmannian $\wt{G}_{n,3}$ of oriented $3$-planes in $\R^n$. We determine the degrees of the indecomposable elements in the cohomology ring. We also obtain an almost complete description of the cohomology ring. This partial description allows us to provide lower and upper bounds on the cup length of $\wt{G}_{n,3}$. As another application, we show that the upper characteristic rank of $\wt{G}_{n,3}$ equals the characteristic rank of $\wt{\gamma}_{n,3}$, the oriented tautological bundle over $\wt{G}_{n,3}$.
\end{abstract}





\section{Introduction}

\hf\hf Grassmann manifolds has been a central object in topology and geometry, especially in the study of classifying spaces and vector bundles. Understanding the cohomology ring, therefore, is of importance. If we consider the Grassmann manifold of $k$-planes in real $n$-dimensional Euclidean space, denoted by $G_{n,k}$, then its cohomology ring with mod $2$ coefficients is well-known by classical results \cite{Bor53} of Borel. Theses manifolds come naturally equipped with a rank $k$ bundle $\gamma_{n,k}\to G_{n,k}$, called the tautological bundle. The associated characteristic classes $w_i$ are the universal Stiefel-Whitney classes. The manifold $G_{n,k}$ has a double cover $\wt{G}_{n,k}$, the Grassmannian of oriented $k$-planes in $\R^n$. This cover is also the universal cover when $n\geq 2k$ and $k\geq 2$. The direct limit of $\wt{G}_{n,k}$ over $n$ is the classifying space for oriented real vector bundles of rank $k$. However, very little is known about the cohomology ring (even with mod $2$ coefficients) of $\wt{G}_{n,k}$ for general values of $n$ and $k$. \\
\hf\hf Due to \cite{KorRu16b} we have a complete description of $H^*(\wt{G}_{n,2};\mathbb{Z}_2).$ In fact, we observe that an alternate (and geometric) proof of the description of $H^\ast(\wt{G}_{n,2};\Z_2)$ for $n$ even is possible using a stronger result of Lai \cite{Lai74} on even-dimensional complex quadrics. We use algebraic topological methods to analyze $H^\ast(\wt{G}_{n,3};\Z_2)$; this is different from the more algebraic methods used in the literature. We introduce the space 
\bgd
W_{2,1}^n:=\left\{(\wt{P},v)\in \wt{G}_{n,2}\times S^{n-1}\,|\,\mbox{$\wt{P}$ and $v$ are orthogonal}\right\}
\edd
and consider the $S^{n-3}$-bundle given by $W_{2,1}^n \to \wt{G}_{n,2}$. This helps us in characterizing the cohomology ring of $W_{2,1}^n$ (cf. Theorem \ref{HW21}). We also analyze the map $\iota^\ast:H^\ast(W^{n+1}_{2,1})\to H^\ast(W^{n}_{2,1})$ (cf. Theorem \ref{wnwnp1}). Exploiting the existence of the $2$-sphere bundle $W_{2,1}^n\xrightarrow{\mathpzc{sp}} \wt{G}_{n,3}$, given by the span of the two plane and the vector, we are able to analyze the cohomology ring of $\wt{G}_{n,3}$.\\
\hf\hf Due to \cite{Kor15} and \cite{PPR17}, we know the degree of the first indecomposable element of $H^*(\wt{G}_{n,3};\mathbb{Z}_2)$ after degree $3$. In this article we determine the complete set of degrees of indecomposables in $H^*(\wt{G}_{n,3};\mathbb{Z}_2)$ (cf. Theorem \ref{fifth}).
\begin{thma}
\textup{(a)} There is one indecomposable in each of the degrees $2,~3,~ 2^t-1$ in $H^*(\wt{G}_{2^t,3})$.\\
\textup{(b)} There is one indecomposable in each of the degrees $2,~3,~ 2^t-4$ in $H^*(\wt{G}_{n,3})$ for $n=2^t-1,~2^t-2,~2^t-3$.\\
\textup{(c)} Let $2^{t-1}<n\leq 2^t-4$. Then the only indecomposables in $H^*(\wt{G}_{n,3})$ are in degrees $2,~3,~3n-2^t-1,~2^t-4$ with one indecomposable in each degree.
\end{thma}
We note that Theorem A (cf. Theorem \ref{fifth} in this article) gives a different proof of Theorem 1.1 of \cite{PPR17}. \\
\hf\hf Let $\wt{w}_2:=\pi^\ast(w_2),\wt{w}_3:=\pi^\ast(w_3)$ denote the 2nd and 3rd Stiefel-Whitney classes of $\wt{\gamma}_{n,3}$, which is the pullback of $\gamma_{n,3}$ via the covering map $\pi:\wt{G}_{n,3}\to G_{n,3}$. We shall call this bundle the oriented tautological bundle. The two classes $\wt{w}_2,\wt{w}_3$ are the indecomposables in degree $2$ and $3$ respectively. Let us denote an indecomposable of degree $i>3$ by $w_i$. We do not have any specific choice for $w_i$ when $i>3.$ However, if we choose any indecomposable element in $H^i(\wt{G}_{n,3};\mathbb{Z}_2)$, these will satisfy some relations stated in the following result (part of Theorem \ref{relations}).
\begin{thmb}
\textup{(1)} Let $n=2^t$, then 
\bgd
H^*(\wt{G}_{2^t,3})\cong \frac{\frac{\mathbb{Z}_2[\wt{w}_2,\wt{w}_3]}{\langle g_{2^t-2}, g_{2^t-1}\rangle}\otimes \mathbb{Z}_2[w_{2^t-1}]}{\langle w_{2^t-1}^2-P w_{2^t-1}\rangle }
\edd
for some $P(\wt{w}_2,\wt{w}_3)\in \pi^*(H^*(G_{2^t,3})).$\\
\textup{(2)} Let $n=2^t-1,~2^t-2,~2^t-3.$ Then 
\bgd
H^*(\wt{G}_{n,3})\cong \frac{\frac{\mathbb{Z}_2[\wt{w}_2,\wt{w}_3]}{\langle g_{n-2},g_{n-1},g_n\rangle}\otimes \mathbb{Z}_2[w_{2^t-4}]}{\langle w_{2^t-4}^2-P_1 w_{2^t-4}-P_2\rangle }
\edd
for some $P_1(\wt{w}_2,\wt{w}_3),P_2(\wt{w}_2,\wt{w}_3)\in\pi^*(H^*(G_{n,3}))$ with $P_2=0$ for $n=2^t-2,~2^t-3.$\\
\textup{(3)} Let $2^{t-1}<n\leq 2^t-4$. Then we have the following relations:\\
\textup{(a)} If $3n-2^t-1<2^t-4$ then
\bgd
w_{3n-2^t-1}^2 = \mathcal{R}_1+w_{3n-2^t-1}\mathcal{R}_2+w_{2^t-4}\mathcal{R}_3,\,\,w_{2^t-4}^2=0
\edd
while if $3n-2^t-1>2^t-4$ then
\bgd
w_{2^t-4}^2 = \mathcal{Q}_1+w_{3n-2^t-1}\mathcal{Q}_2+w_{2^t-4}\mathcal{Q}_3,\,\,w_{3n-2^t-1}^2=0
\edd
where $\mathcal{R}_i,\mathcal{Q}_i$'s are polynomials in $\wt{w}_2$ and $\wt{w}_3$.\\
\textup{(b)} There exists $v_{2^t-8}\in \pi^*(H^{2^t-8}(G_{n,3};\mathbb{Z}_2))$ and $v_{3n-2^t-5}\in \pi^*(H^{3n-2^t-5}(G_{n,3};\mathbb{Z}_2))$ which are Poincar\'{e} dual to $w_{3n-2^t-1}$ and $w_{2^t-4}$ respectively such that
\bgd
v_{2^t-8} w_{2^t-4}=0,\,\,\,\,v_{3n-2^t-5} w_{3n-2^t-1}=0.
\edd
\end{thmb}
We shall prove an extended result (cf. Theorem \ref{relations}) where $v_{2^t-8},v_{3n-2^t-5}$ will be studied extensively.\\ 
\hf\hf As a corollary (see Corollary \ref{cuplength} of Theorem \ref{relations}) of Theorem B, we obtain a lower bound and an upper bound for the cup-length. Recall that $\mathbb{Z}_2$-cup-length of a path-connected space $X$, denoted by $\texttt{cup}(X)$, is defined as the maximum $r$ such that there exist classes $x_1,x_2,\cdots ,x_r\in H^*(X;\mathbb{Z}_2)$ with non-trivial cup product, i.e., $x_1\cup x_2\cup\cdots \cup x_r\neq0$. Fukaya in \cite[Conjecture 1.2]{Fuk08} conjectured about the $\mathbb{Z}_2$-cup-length of oriented Grassmann manifolds $\wt{G}_{n,3}.$ In \cite{Fuk08, Kor15, KorRu16b, Rus17} the authors have proved the conjecture for $n=2^t-1,~2^t,~2^t+1,~2^t+2,~2^t+2^{t-1}+1,~2^t+2^{t-1}+2$ ($t\geq3$). In \cite{PPR17} the authors have given the exact values of $\texttt{cup}(\wt{G}_{n,3})$ for $2^{t-1}<n\leq 2^{t-1}+\frac{2^{t-1}}{3}-1$ and $t\geq 4$. We note that for \cite{Fuk08, PPR17} one has to make a slight adjustment; the notation used in these two papers $\wt{G}_{n,3}$ consists of the oriented $3$-dimensional subspaces of $\mathbb{R}^{n+3}$. The exact value of $\texttt{cup}(\wt{G}_{n,3})$ is not known for $n\in[2^{t-1}+\frac{2^{t-1}}{3},2^t-2]$ ($t\geq 4$). For one sub-interval $[2^{t-1}+2^{t-2},2^t-2]$ of $[2^{t-1}+\frac{2^{t-1}}{3},2^{t}-2]$, an upper bound for $\texttt{cup}(\wt{G}_{n,3})$ is given in \cite[Corollary 3.2]{Rus17}. In the following result (cf. Corollary \ref{cuplength}) we show that the same upper bound will also work for the whole interval $[2^{t-1}+\frac{2^{t-1}}{3},2^t-2]$. We also provide a lower bound for $\texttt{cup}(\wt{G}_{n,3})$. These bounds are consistent with Fukaya's conjecture.
\begin{obsb}
Let $t\geq 4$ and $n\in[2^{t-1}+\frac{2^{t-1}}{3},2^t-2]$. Then the $\mathbb{Z}_2$-cup-length $\texttt{cup}(\wt{G}_{n,3})$ is bounded above by $\frac{3(n-3)}{2}-2^{t-1}+3$ and is bounded below by the maximum of $2^{t-1}-3$ and $\frac{4(n-3)}{3}-2^{t-1}+3.$ 
\end{obsb}
\hf\hf In another direction we have studied characteristic rank of $\wt{G}_{n,k}$ ($k\geq 5$) and upper characteristic rank of $\wt{G}_{n,3}.$ If $X$ is a connected finite CW complex and $\xi$ is a real (finite rank) vector bundle over $X$, recall from \cite{NaTh14} that the characteristic rank of $\xi$ over $X$, denoted by $\texttt{charrank}_X(\xi)$, is by definition the largest integer $k \leq \texttt{dim}(X)$ such that every cohomology class $x\in H^j (X;\mathbb{Z}_2)$, $0\leq  j\leq k,$ is a polynomial in the Stiefel-Whitney classes $w_i(\xi)$. The upper characteristic rank of $X$, denoted by $\texttt{ucharrank}(X)$, is the maximum of $\texttt{charrank}_X (\xi)$ as $\xi$ varies over vector bundles over $X$. In \cite{Kor15} it was shown that 
\bgd
\texttt{charrank}_{\,\wt{G}_{n,k}}(\wt\gamma_{n,k})\geq n-k+1\,\,\textup{for $k\geq 5$},
\edd
where $\wt\gamma_{n,k}$ is the oriented tautological bundle over $\wt{G}_{n,k}$. Our result (cf. Theorem \ref{k5}, Corollary \ref{ucharrank for k=3}) is the following.
\begin{thmc}
{\it If $k\geq 5$ and $n\geq 2k$, then characteristic rank of $\wt{\gamma}_{n,k}$ is at least $n-k+2$.\\
The upper characteristic rank $\texttt{ucharrank}(\wt{G}_{n,3})=\texttt{charrank}(\wt{\gamma}_{n,3})$ for $n\geq 8$.}
\end{thmc}

\label{prelim}\section{Cohomology rings of relevant spaces}

\subsection{Preliminaries}

\hf\hf We start by recalling some known results. The $\mathbb{Z}_2$-cohomology ring of $G_{n,k}$ is classically well-known due to Borel \cite{Bor53}. We assume that $n\geq 2k$ as there is a diffeomorphism between $G_{n,k}$ and $G_{n,n-k}$. There is an isomorphism of rings
\bgd
H^*(G_{n,k};\mathbb{Z}_2)\cong \frac{\mathbb{Z}_2[w_1,\cdots,w_k]}{\langle \bar{w}_{n-k+1},\cdots, \bar{w}_{n}\rangle}
\edd
where $\langle \bar{w}_{n-k+1},\cdots, \bar{w}_{n}\rangle$ is the ideal generated by $\bar{w}_{n-k+1},\cdots,\!\bar{w}_n.$ Here $w_i$ is the $i^{\textup{th}}$ Stiefel-Whitney class of the tautological $k$-plane bundle $\gamma$ over $G_{n,k}$. Moreover, $\bar{w}_i$ is the homogeneous component of $(1+w_1+\cdots + w_k)^{-1}$ in degree $i$. Let $g_i(w_2,\cdots,w_k)$ (abbreviated to $g_i$ here after) be the reduction of $\bar{w}_i$ modulo $w_1$. We also have the following result from \cite{Kor15}, which we shall use in the proofs of our results.
\begin{lmm} \label{Korbas}
\textup{(i)} If $k=3$ then $g_i(w_2,w_3)=0$ if and only if $i=2^r-3$ for some $r\geq3.$\\
\textup{(ii)} If $k\geq5$, then for $i\geq2$, the expression $g_i(w_2,\cdots,w_k)$ is never zero.
\end{lmm}
\hf\hf Recall the following well-known result. Consider the Gysin sequence for the sphere bundle $S^{n-1}\hookrightarrow Y\xrightarrow{p}X$ associated to a real vector bundle $\mathbb{R}^n\hookrightarrow\xi\to X$ (cf. \cite[p.~144]{MiSt}).
\bgd
\cdots \to H^i(X)\xrightarrow{\cup \,w_n}H^{i+n}(X)\xrightarrow{p^*}H^{i+n}(Y)\to H^{i+1}(X)\xrightarrow{\cup\,w_n}H^{i+n+1}(X)\to\cdots
\edd
where $w_n\in H^n(X;\Z_2)$ is the $n^\textup{th}$ Stiefel-Whitney class of $\xi$. It is well known that the above Gysin sequence is the degeneration of the Serre spectral sequence associated to the sphere bundle. The only possible non-zero differential in the spectral sequence is $d^n:E_n^{p,n-1}\to E_n^{p+n,0}$. Hence, we need only compute $d^n:E_n^{0,n-1}\to E_n^{n,0}$. There is a filtration $0\subset F_n^n\subset F_{n-1}^n\subset \cdots\subset F_0^n=H^n(Y)$. We know from \cite{Bor53} that
\bgd
H^n(X)\xrightarrow{i_*} H^n(X;H^0(S^{n-1}))=E_2^{n,0}\to E_{n+1}^{n,0}=E_\infty^{n,0}\hookrightarrow H^n(Y)
\edd
is the map $p^*.$ From the Gysin sequence, we know that the kernel of the map is $\mathbb{Z}_2 w_n$. Therefore, $d^n(1)=w_n$ where $1\in H^0(X;H^{n-1}(S^{n-1}))$ is the generator.\\[0.2cm]
\hf\hf We shall introduce three fibre bundles which will be important in what follows.\\[0.2cm]
(i) {\bf (double cover)}\hf The universal covering map is a principal $\Z_2$-bundle of the form 
\bgd\label{2covgn3}
\mathbb{Z}_2\hookrightarrow \wt{G}_{n,3}\xrightarrow{\pi} G_{n,3}.
\edd
The associated Gysin sequence (with $\Z_2$-coefficients) is written below:\\
\bgd
\cdots \to H^{j-1}(G_{n,3})\xrightarrow{\cup w_1}H^j(G_{n,3})\xrightarrow{\pi^*}H^j(\wt{G}_{n,3})\to H^j(G_{n,3})\xrightarrow{\cup w_1}H^{j+1}(G_{n,3})\to\cdots \label{Gysinw1}
\edd
We have that 
\bgd
\pi^\ast(w_1)=0,\,\,\wt{w}_2:=\pi^\ast(w_2)\neq 0,\,\,\wt{w}_3:=\pi^\ast(w_3)\neq 0.
\edd
For $j+1<n-2$, $H^j(G_{n,3})\xrightarrow{\cup\,w_1} H^{j+1}(G_{n,3})$ is injective, so $H^j(\wt{G}_{n,3})$ consists of all the homogeneous polynomials of degree $j$ built out of $\wt{w}_2$ and $\wt{w}_3$. As $n\geq 6$, $j=2$ then $j+1$ is always less than $n-2$. So $H^2(\wt{G}_{n,3})=\wt{w}_2 \mathbb{Z}_2$. This also implies that for $n\geq 7$, $H^3(\wt{G}_{n,3})=\wt{w}_3\Z_2$. When $n=6$, we still can check that $H^3(\wt{G}_{6,3})=\wt{w}_3\Z_2$. Thus, $H^3(\wt{G}_{n,3})=\wt{w}_3 \mathbb{Z}_2$ for any $n\geq 6$. \\[0.2cm]
\hf\hf In \cite[Theorem 2.1(1)]{Kor15} it has been shown that up to degree at most $n-2$, the map $H^*(G_{n,3})\xrightarrow{\pi^*}H^*(\wt{G}_{n,3})$ is surjective if $n\neq 2^t,~2^t-1,~2^t-2,~ 2^t-3.$
\begin{prpn}\label{P0}
The cohomology map $H^*(G_{n,3})\xrightarrow{\pi^*}H^*(\wt{G}_{n,3})$ is surjective upto degree $n-1$ when $n\neq 2^t,~2^t-1,~2^t-2,~2^t-3.$
\end{prpn}
\begin{proof}
Consider the Gysin sequence \eqref{Gysinw1}. It follows that $H^j(G_{n,3})\xrightarrow{\pi^*}H^j(\wt{G}_{n,3})$ is surjective if and only if the subgroup
\bgd
\text{ker}(H^j(G_{n,3})\xrightarrow{\cup w_1}H^{j+1}(G_{n,3}))\label{kergtogt}
\edd
vanishes. So, we only need to show that when $n\neq2^t,~2^t-1,~2^t-2,~2^t-3$, the subgroup defined in \eqref{kergtogt} vanishes up to degree $n-1$.
Let $x\neq0\in H^{n-1}(G_{n,3})$ be such that $w_1\cup x=0\in H^n(G_{n,3}).$ Then 
\bgd
w_1\cup x=a_1w_1^2\bar{w}_{n-2}+a_2w_2\bar{w}_{n-2}+a_3w_1\bar{w}_{n-1}+a_4\bar{w}_n.
\edd
Reducing the above equation modulo $w_1$, we get $0=a_2w_2g_{n-2}+a_4g_n.$ Here $a_2=0$ implies $a_4g_n=0$. As $n\neq 2^t-3$ implies $g_n\neq 0$ (cf. Lemma 2.3 \cite{Kor15}) we get $a_4=0$. Similarly, $a_4=0$ implies $a_2=0.$ Also, $a_2=a_4=1\in\mathbb{Z}_2$ gives $g_n+w_2g_{n-2}=0$, which gives $w_3g_{n-3}=0$ which is a contradiction as $n\neq 2^t.$ Therefore, $a_2=a_4=0$ and 
\bgd
x=a_1w_1\bar{w}_{n-2}+a_3\bar{w}_{n-1}.
\edd
Thus, $x=0\in H^{n-1}(G_{n,3}),$ which proves that $\text{ker}(H^{n-1}(G_{n,3})\xrightarrow{\cup w_1}H^n(G_{n,3}))$ vanishes. 
\end{proof}
\hf\hf We will also require the following related result.
\begin{lmm}\label{wtGysinw3}
The map
\bgd
\cup\,\wt{w}_3:H^j(\wt{G}_{n,3})\longrightarrow  H^{j+3}(\wt{G}_{n,3})
\edd
is injective if $j\leq n-6$. If $n$ is even then the above map is injective also for $j\leq n-5$.
\end{lmm}
\begin{proof}
\hf It follows from the commutative diagram
\bgd
\xymatrix{
H^j(\wt{G}_{n,3})\ar[r]^-{\cup\,\wt{w}_3} & H^{j+3}(\wt{G}_{n,3})\\
H^j(G_{n,3})\ar[r]^-{\cup\,w_3}\ar[u]^-{\pi^\ast} & H^{j+3}(G_{n,3})\ar[u]_-{\pi^\ast}}
\edd
and \eqref{Gysinw1} that we have an induced commutative diagram 
\bge
\label{Gysinw3}\xymatrix{
H^j(\wt{G}_{n,3})\ar[r]^-{\cup\,\wt{w}_3} & H^{j+3}(\wt{G}_{n,3})\\
H^j(G_{n,3})/(w_1)\ar[r]^-{\cup\,w_3}\ar[u]^-{\bar{\pi}^\ast} & H^{j+3}(G_{n,3})/(w_1)\ar[u]_-{\bar{\pi}^\ast}}
\ede
where the vertical arrows are injective. For $j\leq n-4$ we know that $\pi^\ast:H^j(G_{n,3})\to H^{j}(\wt{G}_{n,3})$ is surjective. Therefore, the map $\bar{\pi}^\ast$ is an isomorphism. As the cohomology ring of $G_{n,3}$, up to degree $n-3$ is represented by polynomials in $w_2$ and $w_3$ with no relations, the map 
\bgd
\cup\,w_3:H^j(G_{n,3})/(w_1)\longrightarrow H^{j+3}(G_{n,3})/(w_1)
\edd
is injective whenever $j+3\leq n-3$, or, equivalently $j\leq n-6$. We can now conclude that the top horizontal map in \eqref{Gysinw3} is injective for $j\leq n-6$. If $n$ is even and $j=n-5$ then $H^{j+3}(G_{n,3})$ has only one relation given by $\bar{w}_{n-2}$. As $n$ is even, the coefficient of $w_2^{n/2-1}$ in $\bar{w}_{n-2}$ is one. Thus, the lower horizontal arrow in \eqref{Gysinw3} is injective, whence the top horizontal arrow is also injective.
\end{proof}
\vspace*{0.2cm}
(ii) {\bf ($2$-sphere bundle)}\hf We shall introduce an intermediate space, for our purposes of studying $\wt{G}_{n,3}$. Towards that we consider the following space
\bgd
W_{2,1}^n:=\left\{(\wt{P},v)\in \wt{G}_{n,2}\times S^{n-1}\,|\,\mbox{$\wt{P}$ and $v$ are orthogonal}\right\}.
\edd
Consider the bundle
\bge\label{s2wn21}
S^2\hookrightarrow W_{2,1}^n\xrightarrow{\mathpzc{sp}} \wt{G}_{n,3},\,\,\,\,\spn(\wt{P},v):= \langle\wt{P},v\rangle
\ede
where $\langle \wt{P},v\rangle$ is the oriented $3$-plane generated by $\wt{P}$ and $v$, in order. This is the sphere bundle of the oriented tautological bundle $\wt{\gamma}_{n,3}\to \wt{G}_{n,3}$. The fiber over an oriented $3$-plane is given by the unit vectors in it, i.e., it is $S^2$. The space $W_{2,1}^n$ is simply connected. Consider the associated Gysin sequence:\\
$$\cdots\to H^{j-3}(\wt{G}_{n,3})\xrightarrow{\cup \wt{w}_3}H^j(\wt{G}_{n,3})\xrightarrow{\spn^\ast}H^j(W_{2,1}^n)\to H^{j-2}(\wt{G}_{n,3})\xrightarrow{\cup \wt{w}_3}H^{j+1}(\wt{G}_{n,3})\to\cdots$$
where $\wt{w}_3\in H^3(\wt{G}_{n,3})$ is the (mod $2$ reduction of the) Euler class of the bundle. As $W_{2,1}^n$ is simply connected, $H^1(W_{2,1}^n)=0$. 
We also conclude that $H^2(W_{2,1}^n)=\spn^\ast(\wt{w}_2) \mathbb{Z}_2,$ where $\spn^\ast(\wt{w}_2)\neq0$. We shall denote $\spn^\ast(\wt{w}_2)$ by $\w_2$.\\[0.3cm]

(iii) {\bf (circle bundle)}\hf To introduce the following bundle, let us recall Stiefel manifolds. The real Stiefel manifold $V_k(\mathbb{R}^n)$ is the space of all orthonormal $k$-tuples in $\mathbb{R}^n$. It is a classical result \cite{Bor53}, due to Borel, that the cohomology ring of $V_k(\mathbb{R}^n)$ with $\mathbb{Z}_2$-coefficients can be described explicitly:
\bge\label{V3n}
H^\ast(V_k(\R^n);\Z_2)\cong \Z_2[a_{n-k},a_{n-k+1},\cdots,a_{n-1}]/I
\ede
with $a_i\in H^i(V_k(\mathbb{R}^n);\Z_2)$ and $I$ being the ideal generated by $a_i^2=a_{2i}$ when $2i\leq n-1$ and $a_i^2=0$ otherwise.
\bgd\label{s1v3n}
S^1\hookrightarrow V_3(\mathbb{R}^n)\xrightarrow{p} W_{2,1}^n,\,\,\,\,p(v_1,v_2,v_3):= ( \langle v_1,v_2\rangle,v_3)
\edd
where $\langle v_1,v_2\rangle$ is the oriented $2$-plane generated by the oriented bases $\{v_1,v_2\}$. In fact, this is a principal $S^1$-bundle with the action of the circle on $V_3(\R^n)$ given by rotating the first two entries $\{v_1,v_2\}$ by the same angle along the plane spanned by the two vectors. The Gysin sequence correspoding to this sphere bundle is 
$$\cdots \to H^{j-2}(W_{2,1}^n)\xrightarrow{\cup E}H^j(W_{2,1}^n)\xrightarrow{p^*}H^j(V_3(\mathbb{R}^n))\to H^{j-1}(W_{2,1}^n)\xrightarrow{\cup E}H^{j+1}(W_{2,1}^n)\to\cdots $$
where $E$, the Euler class mod $2$, is given by $E=\w_2.$

\begin{defn}[Decomposability]
We say that the cohomology ring of a space at a particular degree $r$ is decomposable with respect to a given filtration $F$ if 
\bgd
\bigoplus_{i+j=r, p+q=s,i,j>0}F_p^i\otimes F_q^j\to F_s^r
\edd
is surjective for all $s$. Moreover, we say that the $E_\infty$-page at degree $r$ is decomposable if 
\bgd
\bigoplus_{i+j=r^\prime,p+q=s,(i,p)\neq(0,0),(j,q)\neq(0,0))}E_\infty^{i,p}\otimes E_\infty^{j,q}\to E_\infty^{r^\prime,s}
\edd
is surjective for all $r^\prime+s=r$.
\end{defn}
\begin{prpn}\label{decomposable}
The Serre spectral sequence for cohomology of a fiber bundle comes with a filtration $F$. The cohomology ring of the total space is decomposable at degree $r$ with respect to the filtration $F$ if and only if $E_\infty$-page is decomposable at degree $r$.
\end{prpn}
\begin{proof}
Suppose the cohomology ring is decomposable at degree $r$ with respect to the filtration $F$. Then we have\\
$$\bigoplus_{i+j=r, p+q=s,i,j>0}F_p^i\otimes F_q^j\to F_s^r\mbox{ is surjective for all $s$.}$$
This implies that 
$$\bigoplus_{i+j=r, p+q=s,i,j>0}F_p^i/F_{p+1}^i\otimes F_q^j/F_{q+1}^j\to F_s^r/F_{s+1}^r\mbox{ is surjective for all $s$.}$$
$$\Rightarrow\bigoplus_{i+j=r, p+q=s,i,j>0}E_\infty^{p,i-p}\otimes E_\infty^{q,j-q}\to E_\infty^{s,r-s}\mbox{ is surjective for all $s$.}$$
Thus the $E_\infty$-page at degree $r$ is decomposable.\\
\hf To prove the converse, suppose that the $E_\infty$-page at degree $r$ is decomposable. Choose $x\in F_s^r-F_{s+1}^r$. Then, $x\neq0\in E_\infty^{s,r-s}.$ Then $\bar{x}=\sum_k\bar{x}_k\cup \bar{y}_k$ where $x_k\in F_p^i$, $y_k\in F_q^j$ and $i,j>0$ with $p+q=s$, $i+j=r$. So, $x-\sum_k x_k\cup y_k\in F_{s+1}^r$. Now, by induction $F_{s+1}^r$ is decomposable. So, $x\in F_s^r$ is also decomposable. Therefore, the cohomology ring of the total space is decomposable at degree $r$ with respect to the filtration $F$.
\end{proof}

It follows from the definition that if the cohomology ring is decomposable at degree $r$ with respect to the filtration $F$, then the cohomology ring is decomposable at degree $r$.\\

\subsection{The space $\wt{G}_{n,2}$ and its cohomology ring}

\hf\hf The cohomology ring of $\wt{G}_{n,2}$ with $\Z_2$-coefficients has been characterized completely in \cite{KorRu16a}.
\begin{thm}[Korbas-Rusin]\label{cohg2n}
Let $n\geq 4$ and consider the Grassmannian $\wt{G}_{n,2}$ of oriented $2$-planes in $\R^n$. Let $\wt{w}_2$ denote $\pi^\ast w_2$, where $\pi:\wt{G}_{n,2}\to G_{n,2}$ is the covering map.\\
\textup{(i)} If $n$ is odd then there is an isomorphism of rings
\bgd
H^\ast(\wt{G}_{n,2};\Z_2)\cong \Z_2[\wt{w}_2]/(\wt{w}_2^{\frac{n-1}{2}})\otimes_{\Z_2} \Lambda_{\Z_2}(a_{n-1})
\edd
where $a_{n-1}\in H^{n-1}(\wt{G}_{n,2};\Z_2)$. \\
\textup{(ii)} If $n\equiv 0\,(\textup{mod}\,4)$ then 
\bgd
H^\ast(\wt{G}_{n,2};\Z_2)\cong \Z_2[\wt{w}_2]/(\wt{w}_2^{\frac{n}{2}})\otimes_{\Z_2} \Lambda_{\Z_2}(b_{n-2})
\edd
where $b_{n-2}\in H^{n-2}(\wt{G}_{n,2};\Z_2)$. \\
\textup{(iii)} If $n\equiv 2\,(\textup{mod}\,4)$ then 
\bgd
H^\ast(\wt{G}_{n,2};\Z_2)\cong \frac{\Z_2[\wt{w}_2]/(\wt{w}_2^{\frac{n}{2}})\otimes_{\Z_2} {\Z_2}[b_{n-2}]}{(b_{n-2}^2-\wt{w}_2^{\frac{n-2}{2}}b_{n-2})},
\edd
where $b_{n-2}\in H^{n-2}(\wt{G}_{n,2};\Z_2)$. 
\end{thm}
We note that (ii) and (iii) of the above theorem can be deduced from a much stronger result of Lai \cite{Lai74}. Let $\wt{\gamma}\to \wt{G}_{2n+2,2}$ be the oriented tautological bundle, i.e., the pullback by the covering map of the tautological $2$-plane bundle $\gamma_{2n+2,2}$ over $G_{2n+2,2}$.
\begin{thm}[Lai]\label{cohg2even}
The integral cohomology groups of $\wt{G}_{2n+2,2}$ are isomorphic to those of $\mathbb{CP}^n\times S^{2n}$. Moreover, the cohomology ring is generated by $\wt{\Omega}:=e(\wt{\gamma})\in H^2(\wt{G}_{2n+2,2};\Z)$ and $\kappa\in H^{2n}(\wt{G}_{2n+2,2};\Z)$ with the relations
\bgd
\wt{\Omega}^{n+1}=2\kappa\cup\wt{\Omega},\,\,\kappa\cup\wt{\Omega}^n=(-1)^n,\,\,\wt{\Omega}^{2n}=2(-1)^n,\,\,\kappa\cup\kappa=(1+(-1)^n)/2,\,\,\kappa\cup\Omega=1\in H^{4n}(\wt{G}_{2n+2,2}).
\edd
\end{thm}
{\it Proof of Theorem \ref{cohg2n} (ii),(iii) using Theorem \ref{cohg2even}.}\hf It follows from universal coefficient theorem in cohomology and homology that there is an isomorphism of rings
\bgd
H^\ast(\wt{G}_{2n+2,2};\Z)\otimes\Z_2\xrightarrow{\cong} H^\ast(\wt{G}_{2n+2,2};\Z_2).
\edd
Note that the isomorphism above identifies $\wt{\Omega}\otimes 1$ with $\wt{w}_2$, the second Stiefel-Whitney class of $\wt{\gamma}$. Let $b$ denote $\kappa\otimes 1\in H^{2n}(\wt{G}_{2n+2,2};\Z_2)$. Therefore, the relations among generators in $H^\ast(\wt{G}_{2n+2,2};\Z_2)$ become
\bgd
\wt{w}_2^{n+1}=0,\,\,b\cup\wt{w}_2^n=1,\,\,\wt{w}_2^{2n}=0,\,\,b\cup b=\left\{\begin{array}{rl}
0 & \textup{if $\,2n+2$ $\equiv$ $0$ mod $4$}\\
1 & \textup{if $\,2n+2$ $\equiv$ $2$ mod $4$}
\end{array}\right.
\edd
The third relation is redundant as it follows from the first while the second and fourth combine to give the relation $b^2=0$ or $b^2+b\wt{w}_2^n=0$ depending on whether $2n+2$ is $0$ or $2$ module $4$.$\hfill\square$\\[0.2cm] 
\hf\hf It follows that when $n$ is odd, $H^*(\wt{G}_{n,2})$ has only $2$ indecomposables $\wt{w}_2,a_{n-1}$ in degrees $2$ and $n-1$ respectively while if $n$ is even, $H^*(\wt{G}_{n,2})$ has only $2$ indecomposables $\wt{w}_2,b_{n-2}$ in degree $2$ and $n-2$ respectively. We wish to analyze the map induced, in cohomology with $\Z_2$-coefficients, by the natural inclusion $\wt{i}:\wt{G}_{n,2}\hookrightarrow \wt{G}_{n+1,2}$. 
\begin{prpn}\label{mapg2ng2n+1}
The induced map $\wt{i}^\ast:H^\ast(\wt{G}_{n+1,2};\Z_2)\to H^\ast(\wt{G}_{n,2};\Z_2)$ on the cohomology rings is described by the following action on the generators:
\begin{eqnarray*}
\wt{w}_2\mapsto \wt{w}_2,\,b_{n-1}\mapsto a_{n-1} & & \textup{if $n$ is odd}\\
\wt{w}_2\mapsto \wt{w}_2,\,a_{n}\mapsto \wt{w}_2 b_{n-2} & & \textup{if $n$ is even}.
\end{eqnarray*}
\end{prpn}
\begin{proof}
Consider the ring homomorphism $\wt{i}^*:H^*(\wt{G}_{n+1,2})\to H^*(\wt{G}_{n,2})$. There is map of covering spaces
$$\xymatrix@C=.8cm{
\wt{G}_{n,2} \ar[r]^-{\wt{i}} \ar[d]_-{\pi} & \wt{G}_{n+1,2}\ar[d]^{\pi}\\
G_{n,2} \ar[r]^-{i} & G_{n+1,2}
}
$$
The tautological bundle over $G_{n+1,2}$ pulls back under $i$ to the tautological bundle over $G_{n,2}$. Similar considerations hold for all the other three maps in the diagram above. This implies that $\wt{i}^\ast(\wt{w}_2)=\wt{w}_2$. The above diagram also induces a map between the Gysin sequences associated to the double covers:
$$
\xymatrix@C=.65cm{
\cdots\ar[r] & H^{j-2}(G_{n+1,2}) \ar[d]^-{i^*} \ar[r]^-{\cup w_1} & H^{j-1}(G_{n+1,2}) \ar[d]^-{i^*} \ar[r]^-{\pi^*} & H^{j-1}(\wt{G}_{n+1,2}) \ar[d]^-{\wt{i}^*} \ar[r]^-{\varphi_{n+1}} & H^{j-1}(G_{n+1,2}) \ar[d]^-{i^*} \ar[r]^-{\cup w_1} & H^j(G_{n+1,2})\ar[d]^-{i^*}\ar[r] & \cdots\\
\cdots\ar[r] & H^{j-2}(G_{n,2}) \ar[r]^-{\cup w_1} & H^{j-1}(G_{n,2}) \ar[r]^-{\pi^*} & H^{j-1}(\wt{G}_{n,2}) \ar[r]^-{\varphi_n} & H^{j-1}(G_{n,2}) \ar[r]^-{\cup w_1} & H^j(G_{n,2})\ar[r] & \cdots} 
$$
We now analyze the two cases separately, depending on the parity of $n$. \\
\hf Case 1: Let $n$ be odd, $j=n$ and note that 
\bgd
\pi^\ast(H^{n-1}(G_{n+1,2}))=\Z_2 \wt{w}_2^{\frac{n-1}{2}},\,\,\pi^\ast(H^{n-1}(G_{n,2}))=0
\edd
The diagram of long exact sequences now look like
$$
\xymatrix@C=.65cm{
\Z_2 \wt{w}_2^{\frac{n-1}{2}} \ar[d]\ar[r]^-{i} & H^{n-1}(\wt{G}_{n+1,2}) \ar[d]^-{\wt{i}^*} \ar[r]^-{\varphi_{n+1}} & H^{n-1}(G_{n+1,2}) \ar[d]^-{i^*} \ar[r]^-{\cup w_1} & H^n(G_{n+1,2})\ar[d]^-{i^*}\ar[r] & \cdots\\
0 \ar[r]^-{\pi^*} & H^{n-1}(\wt{G}_{n,2}) \ar[r]^-{\varphi_n} & H^{n-1}(G_{n,2}) \ar[r]^-{\cup w_1} & H^n(G_{n,2})\ar[r] & \cdots} 
$$
The non-zero element $\varphi_{n+1}(b_{n-1})$ is mapped to zero in $H^n(G_{n+1,2})$ when multiplied by $w_1$. We know that
\bgd
H^\ast(G_{n+1,2};\Z_2)\cong\Z_2[w_1,w_2]/(\bar{w}_{n},\bar{w}_{n+1}).
\edd
Thus, $H^j(G_{n+1},2)$ is freely generated by $w_1,w_2$ if $j\leq n-1$. In $H^n(G_{n+1},2)$ we have $\bar{w}_n=0$. As $n$ is odd, $\bar{w}_n=w_1\cup p_{n-1}$ and $\varphi_{n+1}(b_{n-1})=p_{n-1}$. To prove that $\wt{i}^\ast(b_{n-1})=a_{n-1}$, it suffices to show that $i^\ast(p_{n-1})\neq 0$. Note that
\bgd
\bar{w}_n=w_1\bar{w}_{n-1}+w_2\bar{w}_{n-2}=w_1(\bar{w}_{n-1}+w_2q_{n-3})
\edd
as $\bar{w}_{n-2}$ is non-zero and can be written as $w_1 q_{n-3}$. Thus, $p_{n-1}=\bar{w}_{n-1}+w_2q_{n-3}$ and 
\bgd
i^\ast(p_{n-1})=i^\ast(\bar{w}_{n-1})+i^\ast(w_2 q_{n-3})=w_2 q_{n-3}
\edd
where the last equality follows from the fact that $\bar{w}_{n-1}=0\in H^{n-1}(G_{n,2})$ and $i^\ast(\bar{w}_{n-2})=\bar{w}_{n-2}$. Since $w_2q_{n-3}$ has no term of the form $w_1^{n-1}$, it cannot equal $\bar{w}_{n-1}$, whence $w_2 q_{n-3}\neq 0$ in $H^{n-1}(G_{n,2})$. This implies that $i^\ast(p_{n-1})\neq 0$.\\
\hf Case 2: Let $n$ be even and we shall show that $\wt{i}^*(a_n)=\wt{w}_2 b_{n-2}$. Since $H^n(G_{n+1,2})=\Z_2 a_n$ and $H^n(G_{n,2})=\Z_2 \wt{w}_2b_{n-2}$, it suffices to show that $\wt{i}^\ast$ is injective in the following diagram
$$
\xymatrix@C=.65cm{
\cdots\ar[r] & H^n(G_{n+1,2}) \ar[d]^-{i^*} \ar[r]^-{\pi^*} & H^{n}(\wt{G}_{n+1,2}) \ar[d]^-{\wt{i}^*} \ar[r]^-{\varphi_{n+1}} & H^{n}(G_{n+1,2}) \ar[d]^-{i^*} \ar[r]^-{\cup w_1} & H^{n+1}(G_{n+1,2})\ar[d]^-{i^*}\ar[r] & \cdots\\
\cdots\ar[r] &  H^n(G_{n,2}) \ar[r]^-{\pi^*} & H^n(\wt{G}_{n,2}) \ar[r]^-{\varphi_n} & H^n(G_{n,2}) \ar[r]^-{\cup w_1} & H^{n+1}(G_{n,2})\ar[r] & \cdots
} 
$$ 
As both the maps $\pi^\ast$ are zero, the map $\varphi_n$ and $\varphi_{n+1}$ are injective. Let $\varphi_{n+1}(a_n)=p_n$ where $w_1 p_n=\bar{w}_{n+1}$.  Now, 
\bgd
\bar{w}_{n+1}=w_1 \bar{w}_n+w_2 \bar{w}_{n-1}=w_1 \bar{w}_n + w_2 w_1 q_{n-2}=w_1 (w_2 q_{n-2})\in H^{n+1}(G_{n+1,2})
\edd
for some non-zero $q_{n-2}\in H^{n-2}(G_{n,2})$. Therefore, $p_n=w_2 q_{n-2}$ and $i^\ast(p_n)=w_2 q_{n-2}$. It follows from the definition of $\bar{w}_{n-1}$ that
\begin{eqnarray}
q_{n-2} & = & \sum_{r=0}^{\frac{n-2}{2}} {n-1-r \choose r} w_2^r w_1^{n-2-2r}\nonumber \\
& = & \sum_{r=0}^{\frac{n-2}{2}} {n-2-r \choose r} w_2^r w_1^{n-2-2r}+\sum_{r=0}^{\frac{n-2}{2}} {n-2-r \choose r-1} w_2^r w_1^{n-2-2r}\nonumber \\
& = & \bar{w}_{n-2} + \sum_{r=1}^{\frac{n-2}{2}} {n-2-r \choose r-1} w_2^r w_1^{n-2-2r}\nonumber \\
\label{qn-2}& = & \bar{w}_{n-2} + \left(w_2 w_1^{n-4}+\cdots+ \frac{n-2}{2} w_2^{\frac{n-2}{2}}\right).
\end{eqnarray}
It can be verified that if $w_2q_{n-2}=0\in H^n(G_{n,2})$ then $w_2q_{n-2}=\lambda w_1 \bar{w}_{n-1}+\mu \bar{w}_n$, where $\lambda,\mu\in \Z_2$. Note that $w_1^n$ occurs as a term in $\bar{w}_n$ as well as $w_1\bar{w}_{n-1}$ while no such term occurs in $w_2 q_{n-2}$. Therefore, $\lambda=\mu=1$ which implies $q_{n-2}=\bar{w}_{n-2}$. This contradicts the equality \eqref{qn-2}. Thus, $i^\ast(p_n)\neq 0$ and $\wt{i}^*(a_n)=\wt{w}_2 b_{n-2}$.
\end{proof}

\subsection{The space $W_{2,1}^n$ and its cohomology ring}

\hf\hf Let us consider the sphere bundle 
\bgd
S^{n-3}\hookrightarrow W_{2,1}^n\xrightarrow{p}\wt{G}_{n,2},\,\,(\wt{P},v)\mapsto\wt{P}
\edd
Note that if $\wt{\gamma}\to\wt{G}_{n,2}$ denotes the oriented tautological bundle then $W^n_{2,1}$ is the sphere bundle of $\wt{\gamma}^\perp$. Thus, the Euler class of $W_{2,1}^n$ is the Euler class $\Omega:=e(\wt{\gamma}^\perp)\in H^{n-2}(\wt{G}_{n,2};\Z)$. Using \cite{Lai74} we can completely describe the cohomology ring of $W^n_{2,1}$ with $\Z_2$-coefficients.
\begin{thm}\label{HW21}
\textup{(a)} The cohomology ring of $W^{2n+2}_{2,1}$ with $\Z_2$-coefficients is isomorphic to that of $\mathbb{CP}^{n-1}\times S^{2n}\times S^{2n+1}$. There are generators $c_{2n}:=p^\ast b_{2n}, \w_2:=p^\ast \wt{w}_2$ and $c_{2n+1}$ in degrees $2n, 2$ and $2n+1$ respectively, where $\wt{w}_2$ is the reduction mod $2$ of the class $\wt{\Omega}=e(\wt{\gamma})\in H^2(\wt{G}_{2n+2,2};\Z)$, such that
\bgd
H^\ast(W^{2n+2}_{2,1};\Z_2)\cong \Z_2[\w_2]/(\w_2^n)\otimes \Lambda_{\Z_2}(c_{2n})\otimes \Lambda_{\Z_2}(c_{2n+1}).
\edd
\textup{(b)} The cohomology ring of $W^{2n+1}_{2,1}$ with $\Z_2$-coefficients is isomorphic to that of $\mathbb{CP}^{n-1}\times S^{2n-2}\times S^{2n}$. There are generators $d_{2n}:=p^\ast a_{2n}, \w_2:=p^\ast \wt{w}_2$ and $d_{2n-2}$ in degrees $2n, 2$ and $2n-2$ respectively, where $\wt{w}_2$ is the reduction mod $2$ of the class $\wt{\Omega}=e(\wt{\gamma})\in H^2(\wt{G}_{2n+1,2};\Z)$, such that
\bgd
H^\ast(W^{2n+1}_{2,1};\Z_2)\cong \frac{\big(\Z_2[\w_2]/(\w_2^n)\otimes \Lambda_{\Z_2}(d_{2n})\big)[d_{2n-2}]}{\langle d_{2n-2}^2-\lambda \w_2^{n-2}d_{2n}-\mu \w_2^{n-1}d_{2n-2}\rangle}
\edd
for some $\lambda,\mu\in\Z_2$.
\end{thm}
\begin{proof} 
We prove (a) first. It follows from the Serre spectral sequence in (integral cohomology) for $S^{2n-1}\hookrightarrow W_{2,1}^{2n+2}\xrightarrow{p}\wt{G}_{2n+2,2},\,\,(\wt{P},v)\mapsto\wt{P}$ that the only non-zero differential is $d^{2n}$ for the $E_{2n}$-page. Moreover, $d^{2n}$ is multiplication by $\Omega$. Using $\Omega=2\kappa-\wt{\Omega}^n$, it can be verified that the integral cohomology groups of $W^{2n+2}_{2,1}$ are the same as that of $\mathbb{CP}^{n-1}\times S^{2n}\times S^{2n+1}$. In $E_{2n+1}$-page we see that $E_{2n+1}^{0,2n-1}=0$ and $E_{2n+1}^{2,2n-1}=\Z$. Let this integer be generated by an element $c\in H^{2n+1}(W^{2n+2}_{2,1};\Z)$. Similarly, $E_{2n+1}^{4n,0}=0$, i.e., the top cohomology class of $\wt{G}_{2n+2,2}$ is killed by $d^{2n}(\kappa\otimes 1)=\kappa\cup\Omega$. \\
\hf A similar consideration with $\Z_2$-coefficients imply that $H^\ast(W^{2n+2}_{2,1};\Z_2)$ is isomorphic as groups to that of $H^\ast(\mathbb{CP}^{n-1}\times S^{2n}\times S^{2n+1};\Z_2)$. Here the relation $\Omega=2\kappa-\wt{\Omega}^n$ becomes $\Omega=\wt{\Omega}^n$ as we are working with $\Z_2$-coefficients. If $1\otimes 1\in E^{0,2n-1}_{2n}$ then
\bgd
d^{2n}(1\otimes 1)=\Omega=\wt{\Omega}^n=\w_2^n
\edd
and we conclude that $E_\infty^{2n,0}=E_{2n+1}^{2n,0}=\Z_2 p^\ast\kappa$.  Since $E_{2n+1}^{4n,0}=E_\infty^{4n,0}=0$ we conclude that $p^\ast\kappa\cup p^\ast\kappa=0$. Let $c_{2n+1}$ denote the generator of $E_\infty^{2,2n-1}=E_{2n+1}^{2,2n-1}=\Z_2$. It can be shown that the only indecomposables occur in degrees $2, 2n, 2n+1$. Since $H^{4n+2}(W^{2n+2}_{2,1};\Z_2)=0$ it follows that $c_{2n+1}^2=0$. These relations along with Poincar\'{e} duality for $W^{2n+2}_{2,1}$ imply that the cohomology ring is generated by $\w_2,c_{2n},c_{2n+1}$ with the relations
\bgd
\w_2^{n}=0,\,\,c_{2n}^2=0,\,\,c_{2n+1}^2=0.
\edd
This completes the proof of (a).\\
\hf Consider the Serre spectral sequence for the sphere bundle $S^{2n-2}\hookrightarrow W_{2,1}^{2n+1}\xrightarrow{p}\wt{G}_{2n+1,2}$. The only possible non-zero differential is $d^{2n-1}$ which is given by multiplying by (the mod $2$ reduction of) $e(\wt{\gamma}^\perp)\in H^{2n-1}(\wt{G}_{2n+1,2};\Z_2)=0$. Therefore, 
\bgd
E_\infty^{p,q}=H^p(\wt{G}_{2n+1,2})\otimes H^{q}(S^{2n-2})
\edd
and $W^{2n+1}_{2,1}$ has the same cohomology groups as that of $\mathbb{CP}^{n-1}\times S^{2n-2}\times S^{2n}$. According to Theorem \ref{decomposable}, $H^*(W_{2,1}^{2n+1})$ is decomposable in degrees other than $2,2n-2, 2n$. From the ring structure of $E_\infty$ we observe that there are indecomposables only in degree $2, 2n-2, 2n$, denoted by $\w_2,d_{2n-2}$ and $d_{2n}$ where $\w_2=p^\ast\wt{w}_2$ and $d_{2n}=p^\ast(a_{2n})$ are in the image of $p^*(H^*(\wt{G}_{2n+1,2}))$ and are the only indecomposables in these degrees. However, $d_{2n-2}$ is not uniquely determined as an indecomposable. In $H^*(\wt{G}_{2n+1,2})$, $a_{2n}^2=0$ and this implies that $d_{2n}^2=0$. Similarly, it follows from Theorem \ref{cohg2n} (i) that $\w_2^{n-1}\neq 0$ but $\w_2^n=0$. It follows from Poincar\'{e} duality that $\w_2^{n-1} d_{2n} d_{2n-2}$ is the top cohomology class. Note that $\w_2^j d_{2n}\neq \w_2^{j+1} d_{2n-2}$ if $j\leq n-2$; if these were equal then
\bgd
0\neq \w_2^{n-1} d_{2n-2} d_{2n}=\w_2^{n-2-j}(\w_2^{j+1} d_{2n-2})d_{2n}=\w_2^{n-2-j}\w_2^{j} d_{2n}^2=0 
\edd
which is a contradiction. We note that
\bgd
d_{2n-2}^2=\lambda \w_2^{n-2}d_{2n}+\mu \w_2^{n-1}d_{2n-2}
\edd
where $\lambda,\mu\in\Z_2$.

\end{proof}

\hf\hf Let $\iota:W_{2,1}^n\hookrightarrow W_{2,1}^{n+1}$ be the canonical inclusion. We shall compute $\iota^*:H^*(W_{2,1}^{n+1})\to H^*(W_{2,1}^n)$, by evaluating $\iota^\ast$ on the indecomposables, in the following result.
\begin{thm}\label{wnwnp1}
The map $\iota^*:H^*(W_{2,1}^{n+1};\Z_2)\to H^*(W_{2,1}^n;\Z_2)$ induced by $\iota:W_{2,1}^{n}\to W_{2,1}^{n+1}$ is the following:
\bgd
\iota^*:H^*(W_{2,1}^{2k+2})\to H^*(W_{2,1}^{2k+1}),\,\, \w_2\mapsto \w_2,\, c_{2k}\mapsto d_{2k},\,c_{2k+1}\mapsto 0.
\edd
\bgd
\iota^*:H^*(W_{2,1}^{2k+1})\to H^*(W_{2,1}^{2k}),\,\, \w_2\mapsto \w_2,\, d_{2k-2}\mapsto c_{2k-2},\,d_{2k}\mapsto \w_2 c_{2k-2}.
\edd
\end{thm}
\begin{proof}
We consider the following map of fibrations:
$$
\xymatrix@C=.5cm{
S^{n-3}\ar@{^{(}->}[d]\ar[r]^-{i} & S^{n-2}\ar@{^{(}->}[d]\\
W_{2,1}^n\ar[d]^{p}\ar[r]^-{\iota} & W_{2,1}^{n+1}\ar[d]^{p}\\
\wt{G}_{n,2}\ar[r]^-{\wt{i}} & \wt{G}_{n+1,2}
}
$$
It is clear that $\iota^*(\w_2)=\w_2$ because $\w_2$ in $H^*(W_{2,1}^n)$ and $H^*(W_{2,1}^{n+1})$ are the images of the same $\wt{w}_2$ from $H^*(\wt{G}_{n,2})$ and $H^*(\wt{G}_{n+1,2})$ respectively. Now, we consider the even and odd cases separately.\\
\hf Case 1: Let $n=2k+1$ be odd. We know that $H^*(W_{2,1}^{2k+2})$ is generated by $\w_2, c_{2k}=p^\ast(b_{2k}), c_{2k+1}$ while $H^*(W_{2,1}^{2k+1})$ is generated by $\w_2, d_{2k}=p^\ast(a_{2k}), d_{2k-2}$.  It follows from Proposition \ref{mapg2ng2n+1} that $\wt{i}^\ast(b_{2k})=a_{2k}$. This implies that $\iota^\ast(c_{2k})=d_{2k}$. As $H^*(W_{2,1}^{2k+1})$ is non-zero in even degrees only, we have $\iota^*(c_{2k+1})=0$.\\
\hf Case 2: Let $n=2k$ be even. As $H^*(W_{2,1}^{2k+1})$ is generated by $\w_2, d_{2k-2}, d_{2k}=p^\ast(a_{2k})$ and $H^*(W_{2,1}^{2k})$ is generated by $\w_2, c_{2k-2}=p^\ast(b_{2k-2}), c_{2k-1}$ and $\wt{i}^\ast(a_{2k})=\wt{w}_2 b_{2k-2}$ (cf. Proposition \ref{mapg2ng2n+1}), it follows that $\iota^*(d_{2k})=\w_2 c_{2k-2}$. It remains to compute $\iota^*(d_{2k-2})$. To prove this, we need to consider the map between Serre spectral sequences (in $\Z_2$-cohomology) related to the following map between fibrations (cf. \eqref{s2wn21}).\\
\bge\label{wnwn+1}
\xymatrix@C=.5cm{
S^2\ar@{^{(}->}[d]\ar[r]^-{i} & S^{2}\ar@{^{(}->}[d]\\
W_{2,1}^n\ar[d]^{\mathpzc{sp}}\ar[r]^-{\iota} & W_{2,1}^{n+1}\ar[d]^{\mathpzc{sp}}\\
\wt{G}_{n,3}\ar[r]^-{\wt{i}} & \wt{G}_{n+1,3}
}
\ede
Note that either of the fibre bundles are unit sphere bundles of the oriented tautological $3$-plane bundle $\wt{\gamma}_{m,3}$ over the base $ \wt{G}_{m,3}$. Therefore, in both the spectral sequences the only non-zero differential is $d^3$, which is given by multiplication by cup product with $\wt{w}_3\in H^3(\wt{G}_{n,3})$. This class is actually the mod $2$ reduction of the Euler class of $\wt{\gamma}$. \\
\hf We first assume that $n=2k\neq 2^t-2$ for any $t$. Let us consider the Serre spectral sequence for the fibre bundle $S^2\hookrightarrow W_{2,1}^{2k}\xrightarrow{\mathpzc{sp}}\wt{G}_{2k,3}$. Note that $H^{2k-2}(W_{2,1}^{2k})\cong\Z_2$ generated by $c_{2k-2}$. The only non-zero differential is in the $E_3$-page, given by cup product with $\wt{w}_3$. Thus, the $E_2$-page equals the $E_3$-page and $E_3^{j,2}=H^j(\wt{G}_{2k,3})\otimes H^2(S^2)$. Let $\mathbbm{1}$ denote the generator of $H^2(S^2)\cong\Z_2$. Thus, in the $E_3$-page
\bgd
d^3:E_3^{j,2}\to E_3^{j+3,0},\,\,\alpha\otimes \mathbbm{1}\mapsto \alpha\cup\wt{w}_3
\edd
is injective whenever $j\leq 2k-5$ due to Lemma \ref{wtGysinw3}. Now consider $\bar{w}_{2k-1}$; it is a polynomial of degree $2k-1$ in $w_1, w_2$ and $w_3$. As $g_{2k-1}$ is the reduction of $\bar{w}_{2k-1}$ modulo $w_1$, $g_{2k-1}=\pi^\ast(\bar{w}_{2k-1})=0\in H^{2k-1}(\wt{G}_{2k,3})$. Observe that $g_{2k-1}$, as a polynomial in $\wt{w}_2,\wt{w}_3$ is non-zero (cf. Lemma \ref{Korbas} (i)) because otherwise $2k-1=2^t-3$, which is impossible by our assumption on $n$. Moreover, as $2k-1$ is odd, $g_{2k-1}$ cannot have $\wt{w}_2^k$ as a term. We may write $g_{2k-1}=\wt{w}_3 \wt{P}_{2k-4}$, where $\wt{P}_{2k-4}=f(\wt{w}_2,\wt{w}_3)$ is non-zero polynomial in $\wt{w}_2,\wt{w}_3$ of degree\footnote{The degree of $f$ is at most $n/2-2$ as a polynomial in two variables. However, here the variables $x$ and $y$ have internal degrees. Therefore, the degree of $x^i y^j$ is not $i+j$ but $i |x|+j|y|$, where $|x|,|y|$ denotes the degrees of $x$ and $y$ respectively.} $2k-4$. As $H^{2k-4}(G_{2k,3})$ is generated freely by polynomials in $w_1,w_2,w_3$ of degree $2k-4$, we conclude that $P_{2k-4}:=f(w_2,w_3)\neq 0$ in $H^{2k-4}(G_{2k,3})$. Now $\bar{\pi}^\ast:H^{j}(G_{2k,3})/(w_1)\to H^j(\wt{G}_{2k,3})$ is injective. Therefore, $\wt{P}_{2k-4}=\bar{\pi}^\ast(P_{2k-4})\neq 0$ and 
\bgd
E_\infty^{2k-4,2}=E_4^{2k-4,2}=\text{ker}(E_3^{2k-4,2}\xrightarrow{\cup\, \wt{w}_3}E_3^{2k-1,0})=\Z_2,
\edd
generated by $\wt{P}_{2k-4}\otimes\mathbbm{1}$. Therefore, $E_\infty^{2k-2,0}=0$ and 
\bgd
H^{2k-2}(W^{2k}_{2,1})\cong E_\infty^{2k-2,0}\oplus E_\infty^{2k-4,2}=E_\infty^{2k-4,2}.
\edd
Thus, the element $\wt{P}_{2k-4}\otimes\mathbbm{1}$ in $E_\infty^{2k-4,2}$ is identified with $c_{2k-2}\in H^{2k-2}(W^{2k}_{2,1})$. \\
\hf We now consider the Serre spectral sequence related to the fibre bundle $S^2\hookrightarrow W_{2,1}^{2k+1}\xrightarrow{\mathpzc{sp}}\wt{G}_{2k+1,3}.$ Like the last analysis, we conclude that $E_\infty^{j,2}=0$ for $j\leq (2k+1)-6=2k-5$. Thus, the first non-zero $E_\infty^{j,2}$ occurs when $j=2k-4$. Even here, $E_\infty^{2k-4,2}=P_{2k-4}\otimes \mathbbm{1}$ and this element is identified with either $d_{2k-2}$ or $d_{2k-2}+\w_2^{k-1}$. The map between spectral sequences implies that $E_\infty^{2k-4,2}\to E_\infty^{2k-4,2}$ is non-zero. It follows that
\bgd
\iota^*:H^{2k-2}(W_{2,1}^{2k+1})\to H^{2k-2}(W_{2,1}^{2k}),\,\,d_{2k-2}\to c_{2k-2}.
\edd
\hf We now assume that $n=2k=2^t-2$ for some $t\geq 3$. Then $H^{2k-2}(\wt{G}_{2k,3})$ has the first indecomposable element after degree $3$. This element will be $c_{2k-2}+\mu \w_2^{k-1}$ for some $\mu\in\mathbb{Z}_2$. Similarly, $H^{2k-2}(\wt{G}_{2k+1,3})$ has got the first indecomposable element after degree $3$. This element will be $d_{2k-2}+\lambda  \w_2^{k-1}$ for some $\lambda \in\mathbb{Z}_2.$ From the naturality of the Gysin sequence, we find that $\iota^*(d_{2k-2})+ \lambda \w_2^{k-1}=c_{2k-2}+\kappa \w_2^{k-1}$ for some $\kappa\in\mathbb{Z}_2$. So, $\iota^*(d_{2k-2})=c_{2k-2}+(\kappa+\lambda) \w_2^{k-1}=c_{2k-2}$.
\end{proof}

\section{Main Results \& Applications}

\subsection{Main Theorems}\label{mthms}

\hf\hf We shall start by discussing some relations in the cohomology ring of $G_{n,3}$.
\begin{lmm} \label {not so}
For $2^{t-1}< n\leq 2^t$, we have $w_2^{n-2^{t-1}}g_{n-3}=g_{3n-2^t-3}+P,$ where $P$ is a polynomial in $w_2,~g_i$'s with $i\geq n-1$ and each monomial contains exactly one $g_i$ with degree $1.$ 
\end{lmm}
\begin{proof}
We shall prove this by induction. When $n=2^{t-1}+2$, we have
\begin{eqnarray*}
w_2^2g_{2^{t-1}-1} & = & w_2(g_{2^{t-1}+1}+w_3g_{2^{t-1}-2})\\
& = & g_{2^{t-1}+3}+w_3g_{2^{t-1}}+w_2w_3g_{2^{t-1}-2}\\
& = & g_{2^{t-1}+3}+w_3^2g_{2^{t-1}-3}+2w_2w_3g_{2^{t-1}-2}\\
& = & g_{2^{t-1}+3}
\end{eqnarray*}
as $g_{2^{t-1}-3}=0$. Note that $w_2 g_{2^{t-1}-2}=g_{2^{t-1}}.$ Let us assume that the statement is true for $n$. We shall prove it for $n+2$. Observe that
\bgd
w_3^2g_{3n-2^t-3} = w_3^2(w_2^{n-2^{t-1}}g_{n-3}+P) = w_2^{n-2^{t-1}}w_3(g_n+w_2g_{n-2})+w_3P_1,
\edd
where $P_1$ is a polynomial in $w_2, g_i$ with $i\geq n$ and each monomial contains exactly one $g_i$ with degree $1$. Hence,
\begin{eqnarray*}
w_3^2g_{3n-2^t-3} & = & w_2^{n-2^{t-1}}(g_{n+3}+w_2g_{n+1}+w_2g_{n+1}+w_2^2g_{n-1})+P_2\\
& = & w_2^{n-2^{t-1}}(g_{n+3}+w_2^2g_{n-1})+P_2
\end{eqnarray*}
where $P_2$ is a polynomial in $w_2, g_i$ with $i\geq n+1$ and each monomial contains exactly one $g_i$ with degree $1$. Therefore, 
\begin{eqnarray*}
w_2^{n-2^{t-1}+2}g_{n-1} & = & w_3^2g_{3n-2^t-3}+w_2^{n-2^{t-1}}g_{n+3}+P_2\\
& = & g_{3n-2^t+3}+w_2^2g_{3n-2^t-1}+w_2^{n-2^{t-1}}g_{n+3}+P_2\\
& = & g_{3n-2^t+3}+P_3
\end{eqnarray*}
where $P_3$ is a polynomial in $w_2, g_i$ with $i\geq n+1$ and each monomial contains exactly one $g_i$ with degree $1$ (as $n\geq 2^{t-1}+1$ implies $3n-2^t-1\geq n+1$). 
\end{proof}
\hf\hf Consider the Serre spectral sequence related to the sphere bundle $S^2\hookrightarrow W_{2,1}^n\xrightarrow{\mathpzc{sp}}\wt{G}_{n,3}$. Here $E_2^{*,j}=0$ if $j\neq2$, and the differentials $d^i=0$ if $i\neq3$ and $d^3:E_3^{0,2}\to E_3^{3,0}$ is given by cup product with $w_3$. Note that $E_2^{i,j}\neq 0$ implies $j=0,2$. \\
\hf\hf Let $n$ be even and $n\neq 2^t,~2^t-2$. Since $n-1$ is odd, if $i\leq n-3$ then $\text{ker}(H^i(\wt{G}_{n,3})\xrightarrow{\cup \wt{w}_3}H^{i+3}(\wt{G}_{n,3}))$ is non-zero only when $i=n-4$ and $i=n-3$. In each case the dimension of the kernal is $1$; call these elements $p_{n-4}$ and $p_{n-3}$ respectively, where 
\bge\label{pneven}
\wt{w}_3 p_{n-4} = g_{n-1},\,\,\,\wt{w}_3 p_{n-3} = g_n +\wt{w}_2 g_{n-2}.
\ede
So, $E_\infty^{n-4,2}\neq0$ and $E_\infty^{n-3,2}\neq0$ and as upto degree $n-1$, $H^*(G_{n,3})\xrightarrow{\pi^*}H^*(\wt{G}_{n,3})$ is surjective, we have $c_{n-2}\not\in F_{n-2}^{n-2}$ and $c_{n-1}\not\in F_{n-1}^{n-1}.$ So, $E_\infty^{n-4,2}=\langle\bar{c}_{n-2}\rangle=\langle p_{n-4}\otimes \mathbbm{1}\rangle$ and $E_\infty^{n-3,2}=\langle\bar{c}_{n-1}\rangle=\langle p_{n-3}\otimes \mathbbm{1}\rangle.$\\
Similarly, when $n$ is odd and $n\neq 2^t-1,~2^t-3$ we have that $d_{n-3}\not\in F_{n-3}^{n-3}$ and $d_{n-1}\not\in F_{n-1}^{n-1}$. We also have $E_\infty^{n-5,2}=\langle\bar{d}_{n-3}\rangle=\langle p_{n-5}\otimes \mathbbm{1}\rangle$ and $E_\infty^{n-3,2}=\langle\bar{\w}_2 \bar{d}_{n-3},\bar{d}_{n-1}\rangle=\langle \wt{w}_2 p_{n-5}\otimes \mathbbm{1}, ~p_{n-3}\otimes \mathbbm{1}\rangle$ where 
\bge\label{pnodd1}
\wt{w}_3 p_{n-5} = g_{n-2},\,\,\,\wt{w}_3 p_{n-3} = g_n.
\ede
\hf\hf Recall \eqref{wnwn+1} in the proof of Theorem \ref{wnwnp1} that we get maps $E_\infty^{*,*}(W_{2,1}^{n+1})\xrightarrow{\iota^*}E_\infty^{*,*}(W_{2,1}^n).$ In fact, as is implied by Theorem \ref{wnwnp1} implies, when $n$ is odd
\bge\label{barnodd}
\bar{\w}_2\to \bar{\w}_2, ~\bar{c}_{n-1}\to\bar{d}_{n-1}, ~\bar{c}_n\to0
\ede
while when $n$ is even
\bge\label{barneven}
\bar{\w}_2\to\bar{\w}_2,~\bar{d}_{n-2}\to\bar{c}_{n-2},~\bar{d}_n\to \bar{\w}_2 \bar{c}_{n-2}.
\ede
Let $n$ be odd such that $n\neq 2^t-1,~2^t-3$. As $E_2^{n-3,2}(W_{2,1}^{n+1})\xrightarrow{\iota^*}E_2^{n-3,2}(W_{2,1}^n)$, we observe that $p_{n-3}\otimes \mathbbm{1}\in E_\infty^{n-3,2}$ is $\bar{d}_{n-1}$. Again we know that $\wt{w}_2 p_{n-5}\otimes \mathbbm{1}\in E_\infty^{n-3,2}$ is $\bar{\w}_2 \bar{d}_{n-3}$. Therefore,
\bge\label{barnoddd}
E_\infty^{n-3,2}\ni q_{n-3}\otimes \mathbbm{1} = \bar{d}_{n-1}+\bar{\w}_2 \bar{d}_{n-3},
\ede
where 
\bge\label{pnodd2}
q_{n-3}=p_{n-3}+\wt{w}_2 p_{n-5}.
\ede
\begin{prpn}\label{P1}
Let $2^{t-1}<n\leq 2^t-4$ and consider the Serre spectral sequence regarding the sphere bundle $S^2\hookrightarrow W_{2,1}^n\xrightarrow{\mathpzc{sp}}\wt{G}_{n,3}$.\\
\textup{(i)} If $n$ is even then the least integer $i$ such that $\w_2^i c_{n-2}\in F_{2i+n-2}^{2i+n-2}$ is $2^{t-1}-\frac{n}{2}-1$ and the least integer $i$ such that $\w_2^i c_{n-1}\in F_{2i+n-1}^{2i+n-1}$ is $n-2^{t-1}.$ \\
\textup{(ii)} If $n$ is odd then \\
\hf\hf \textup{(a)} the least integer $i$ such that $\w_2^i d_{n-1}\in F_{2i+n-1}^{2i+n-1}$ is $2^{t-1}-\frac{n+3}{2}$,\\
\hf\hf \textup{(b)} the least integer $i$ such that $\w_2^i d_{n-3}\in F_{2i+n-3}^{2i+n-3}$ is $\max\{2^{t-1}-\frac{n+1}{2},n-2^{t-1}+1\}$,\\
\hf\hf \textup{(c)} the least integer $i$ such that $\w_2^{i+1}d_{n-3}+\w_2^i d_{n-1}\in F_{2i+n-1}^{2i+n-1}$ is $n-2^{t-1}.$
\end{prpn}
\begin{proof}
Consider the ring $H^*(W_{2,1}^{2^t-4})$ generated by $\w_2, ~c_{2^t-6},~c_{2^t-5}$. Let $p_{2^t-8}\in H^{2^t-8}(\wt{G}_{2^t-4,3})$ be such that $\wt{w}_3 p_{2^t-8}=g_{2^t-5}.$ Then 
\bgd
\langle\bar{c}_{2^t-6}\rangle=E_\infty^{2^t-8,2}=\langle p_{2^t-8}\otimes \mathbbm{1}\rangle.
\edd
As $g_{2^t-3}=0$ (follows from Lemma 2.3 of \cite{Kor15}) we have
\bgd
0=g_{2^t -3}=\wt{w}_2 g_{2^t-5}+\wt{w}_3 g_{2^t-6} =\wt{w}_2 \wt{w}_3 p_{2^t-8}+\wt{w}_3 g_{2^t-6} \Longrightarrow \wt{w}_2 p_{2^t-8}+g_{2^t-6} = 0.
\edd
Therefore, $\wt{w}_2 p_{2^t-8} = 0\in H^*(\wt{G}_{2^t-4,3})$ and $\wt{w}_2 p_{2^t-8}\otimes \mathbbm{1} = 0\in E_2^{2^t-6,2}$. This implies that $\bar{\w}_2 \bar{c}_{2^t-6}=0$ in $E_4^{2^t-4,2}$. \\[0.2cm]
{\bf Observation A1 ($n$ even)}: \vspace{-0.6cm}\begin{center}$\w_2^{2^{t-1}-\frac{n}{2}-1}c_{n-2}\in F_{2^t-4}^{2^t-4}$\end{center}
It follows by an iterated application of \eqref{barneven} and \eqref{barnodd} successively that $\bar{\w}_2^{1+s}\bar{c}_{n-2}=0\in E_{\infty}^{2^t-6,2}$ where $2s=2^t-4-n$. Therefore, $\bar{\w}_2^{2^{t-1}-\frac{n}{2}-1}\bar{c}_{n-2}=0\in E_\infty^{2^t-6,2}$, whence the claim follows.\\[0.2cm]
{\bf Observation A2 ($n$ odd)}: \vspace{-0.6cm}\bgc$\w_2^{2^{t-1}-1-\frac{n+1}{2}}d_{n-1}\in F_{2^t-4}^{2^t-4}$\edc
An iterated application of \eqref{barneven} and \eqref{barnodd} implies that $\bar{\w}_2^{1+s}\bar{d}_{n-1}=0\in E_{\infty}^{2^t-6,2}$ where $2s=2^t-5-n$. Therefore, $\bar{\w}_2^{2^{t-1}-1-\frac{n+1}{2}}\bar{d}_{n-1}=0\in E_\infty^{2^t-6,2}$, whence the claim follows. \\[0.2cm]
{\bf Observation B1 ($n$ even)}: \vspace*{-0.6cm}\bgc$\w_2^{n-2^{t-1}}c_{n-1}\in F_{3n-2^t-1}^{3n-2^t-1}$\edc
As $E_\infty^{n-3,2}$ is generated by $p_{n-3}\otimes \mathbbm{1}$, where $\wt{w}_3 p_{n-3}=\wt{w}_2g_{n-2}+g_n$, we conclude that $p_{n-3}=g_{n-3}$. So, we observe from Lemma \ref{not so} that $\wt{w}_2^{n-2^{t-1}}g_{n-3}=0$ in $H^*(\wt{G}_{n,3})$, whence
\bgd
\wt{w}_2^{n-2^{t-1}}p_{n-3}\otimes \mathbbm{1}=0\in E_\infty^{3n-2^t-3,2}.
\edd
Therefore, $\bar{\w}_2^{n-2^{t-1}}\bar{c}_{n-1}=0\in E_{\infty}^{3n-2^t-3,2}$ and this proves the claim. \\[0.2cm]
{\bf Observation B2 ($n$ odd)}: \vspace{-0.6cm}\bgc$\w_2^{n-2^{t-1}}(d_{n-1}+\w_2 d_{n-3})\in F_{3n-2^t-1}^{3n-2^t-1}$\edc
As $q_{n-3}\otimes \mathbbm{1}\neq 0\in E_\infty^{n-3,2}$, where $\wt{w}_3 q_{n-3}=\wt{w}_2 g_{n-2}+g_n$, it follows that $q_{n-3}=g_{n-3}$. As in (B1) above, we see that $\wt{w}_2^{n-2^{t-1}}g_{n-3}=0$ in $H^*(\wt{G}_{n,3})$, whence 
\bgd
\w_2^{n-2^{t-1}}p_{n-3}\otimes \mathbbm{1}=0\in E_\infty^{3n-2^t-3,2}.
\edd
As we did for (B1), we can show that using \eqref{barnoddd} that the claim follows. \\[0.2cm]
{\bf Observation C1 ($n$ even)}: \vspace{-0.6cm}\bgc$\w_2^i c_{n-2}\not\in F_{2i+n-2}^{2i+n-2}\,\,\textup{for all}\,\,i\leq 2^{t-1}-\frac{n}{2}-2$\edc
We have seen in (B1) that $\w_2^{n-2^{t-1}}c_{n-1}\in F_{3n-2^t-1}^{3n-2^t-1}$. We also know that $\w_2^{\frac{n}{2}-2}c_{n-2}c_{n-1}\not\in F_{3n-7}^{3n-7}$ as $H^{3n-7}(\wt{G}_{n,3})=0$. Thus, $\w_2^{\frac{n}{2}-2}c_{n-2}c_{n-1}=(\w_2^{n-2^{t-1}}c_{n-1})(\w_2^{2^{t-1}-\frac{n}{2}-2}c_{n-2})$ implies $\w_2^{2^{t-1}-\frac{n}{2}-2}c_{n-2}\not\in F_{2^t-6}^{2^t-6}$ and the claim follows. \\[0.2cm]
{\bf Observation C2 ($n$ odd)}: \vspace{-0.6cm}\bgc$\w_2^i d_{n-1}\not\in F_{2i+n-1}^{2i+n-1}\,\,\textup{for all}\,\,i\leq 2^{t-1}-\frac{n+3}{2}-1.$\edc \bgc $\w_2^i d_{n-3}\not\in F_{2i+n-3}^{2i+n-3}\,\,\textup{for all}\,\,i\leq 2^{t-1}-\frac{n+3}{2}$\edc
It follows from the cohomology ring as described in Theorem \ref{HW21} that
\bgc
$(\w_2^{\frac{n-3}{2} - n +2^{t-1}} d_{n-3})(\w_2^{n-2^{t-1}}d_{n-1})=(\w_2^{\frac{n-3}{2} - n +2^{t-1}} d_{n-3})(\w_2^{n-2^{t-1}}d_{n-1}+\w_2^{n-2^{t-1}+1}d_{n-3})=\w_2^{\frac{n-3}{2}}d_{n-1}d_{n-3}$.
\edc
The second claim now follows from (B2) and $\w_2^{\frac{n-3}{2}}d_{n-1}d_{n-3}\not\in F_{3n-7}^{3n-7}$ (as $H^{3n-7}(\wt{G}_{n,3})=0$). Theorem \ref{HW21} implies that
\bgc
$(\w_2^{\frac{n-3}{2} - n +2^{t-1}-1} d_{n-1})(\w_2^{n-2^{t-1}+1}d_{n-3}+\w_2^{n-2^{t-1}}d_{n-1})=\w_2^{\frac{n-3}{2}}d_{n-1}d_{n-3}$.
\edc
Combining this with (B2) we conclude the first claim.  \\[0.2cm]
{\bf Observation D1 ($n$ even)}: \vspace{-0.6cm}\bgc $\w_2^i c_{n-1}\not\in F_{2i+n-1}^{2i+n-1}\,\,\textup{for all}\,\,i\leq n-2^{t-1}-1$\edc
We have shown in (A1) that $\bar{\w}_2^{1+s}\bar{c}_{n-2}=0\in E_{\infty}^{2^t-6,2}$, where $2s=2^t-4-n$. Therefore, $\w_2^{1+s}c_{n-2}\in F_{2^t-4}^{2^t-4}.$ We see that $\w_2^{\frac{n}{2}-2}c_{n-2}c_{n-1}\not\in F_{3n-7}^{3n-7}$ because $H^{3n-7}(\wt{G}_{n,3})=0.$ Therefore, 
\bgd
\w_2^{\frac{n}{2}-2}c_{n-2}c_{n-1}=(\w_2^{1+s}c_{n-2})(\w_2^{\frac{n}{2}-3-s}c_{n-1})
\edd
implies our claim. \\[0.2cm]
{\bf Observation D2 ($n$ odd)}: \vspace{-0.6cm}\bgc $\w_2^i d_{n-3}\not\in F_{2i+n-3}^{2i+n-3}\,\,\textup{for all}\,\,i\leq n-2^{t-1}$\edc
\bgc $\w_2^i d_{n-3}+\w_2^{i-1}d_{n-1}\not\in F_{2i+n-3}^{2i+n-3}\,\,\textup{for all}\,\,i\leq n-2^{t-1}$\edc
We have seen in (A2) that $\w_2^{2^{t-1}-1-\frac{n+1}{2}}d_{n-1}\in F_{2^t-4}^{2^t-4}$. Therefore, 
\bgd
\w_2^{\frac{n-3}{2}}d_{n-3}d_{n-1}=(\w_2^{2^{t-1}-1-\frac{n+1}{2}}d_{n-1})(\w_2^{n-2^{t-1}}d_{n-3})=(\w_2^{2^{t-1}-1-\frac{n+1}{2}}d_{n-1})(\w_2^{n-2^{t-1}}d_{n-3}+\w_2^{n-2^{t-1}-1}d_{n-1})
\edd
implies that $\w_2^{n-2^{t-1}}d_{n-3}\not\in F_{3n-2^t-3}^{3n-2^t-3}$ and $\w_2^{n-2^{t-1}}d_{n-3}+\w_2^{n-2^{t-1}-1}d_{n-1}\not\in F_{3n-2^t-3}^{3n-2^t-3}.$ Both claims now follow. \\[0.2cm]
Now (A1), (C1) prove the first part of (i) while (B1), (D1) prove the second part of (i). Similarly, (A2),(C2) prove (ii)(a) while (B2),(D2) imply (ii)(b) and (ii)(c).
\end{proof}
\hf\hf Let $F_s^r$ be the filtration of $H^r(W_{2,1}^n)$ obtained with the Serre spectral sequence. It follows that $F_0^r=F_{r-2}^r$. We conclude that $F_{i-2}^i\otimes F_{j-2}^j\to F_{i+j-2}^{i+j}.$ We denote by $\mathcal{P}_r$ the following
\bge\label{P_r}
\mathcal{P}_r:=\frac{\text{Image of }\underset{q+j\leq r, p\neq0, i\neq0, p+i=r}{\oplus}F_q^p\otimes F_j^i \text{ in }F_r^r}{\text{Image of }\underset{p+i=r,p\neq0, i\neq0}{\oplus}F_p^p\otimes F_i^i \text{ in }F_r^r}.
\ede
It is seen that $P_r$ is a quotient of a subspace of $F_r^r$. There is a natural map 
\bge\label{spstar}
\frac{H^r(\wt{G}_{n,3})}{H^+(\wt{G}_{n,3})\cdot H^+(\wt{G}_{n,3})}\xrightarrow{\mathpzc{sp}^*} \frac{F_r^r}{\text{Image of }\underset{p+i=r,p\neq0,i\neq0}{\oplus}F_p^p\otimes F_i^i\text{ in }F_r^r}
\ede
Moreover, when $r\geq n$
\bgd
F_r^r = \big(\textup{Image of }\underset{q+j\leq r, p\neq0, i\neq0}{\oplus}F_q^p\otimes F_j^i \textup{ in }F_r^r\big).
\edd
\begin{lmm} \label{important}
The map $\mathpzc{sp}^\ast$ as defined in \eqref{spstar} is a vector space isomorphism for $r>3$. 
\end{lmm}
\begin{proof}
It is clear that the map $\mathpzc{sp}^*$ is a linear morphism. The fact that $\mathpzc{sp}^*$ is onto follows from  $\mathpzc{sp}^*(H^r(\wt{G}_{n,3}))=F_r^r$. Let $x\in H^*(\wt{G}_{n,3})$ be such that $\mathpzc{sp}^*(x)=0$. It follows that
\bgd
\mathpzc{sp}^*(x)\in \textup{image of }\underset{p+i=r,p\neq0,i\neq0}{\oplus}F_p^p\otimes F_i^i\textup{ in }F_r^r.
\edd
As $p,i<r$, so $H^p(\wt{G}_{n,3})\xrightarrow{\mathpzc{sp}^*}F_p^p$ and $H^i(\wt{G}_{n,3})\xrightarrow{\mathpzc{sp}^*}F_i^i$ are surjective maps. Therefore, there exists $P\in H^+(\wt{G}_{n,3})\cdot H^+(\wt{G}_{n,3})$ such that $\mathpzc{sp}^*(x)=\mathpzc{sp}^*(P)$. Hence, $\mathpzc{sp}^*(x-P)=0$ and this implies that $x-P=w_3 Q$ for some $Q\in H^*(\wt{G}_{n,3})$. As $x$ has degree $r>3$, the degree of $Q$ is at least $1$. As a consequence, $x=P+w_3 Q \in H^+(\wt{G}_{n,3})\cdot H^+(\wt{G}_{n,3}).$
\end{proof}

In the following theorem we shall compute all the indecomposables in $H^*(\wt{G}_{n,3})$.
\begin{thm}\label{fifth}
\textup{(a)} The only indecomposables in $H^*(\wt{G}_{2^t,3})$ are in degrees $2,~3,~2^t-1$ with one indecomposable in each degree.\\
\textup{(b)} The only indecomposables in $H^*(\wt{G}_{n,3})$ for $n=2^t-1,~2^t-2,~2^t-3$ are in degrees $2,~3,~2^t-4$ with one indecomposable in each degree.\\
\textup{(c)} Let $2^{t-1}<n\leq 2^t-4$. Then the only indecomposables in $H^*(\wt{G}_{n,3})$ are in degrees $2,~3,~3n-2^t-1,~2^t-4$ with one indecomposable in each degree.
\end{thm}
\begin{proof}
(a) By \cite{Kor15} we know that apart from $\wt{w}_2,\wt{w}_3$ the first indecomposable in $H^*(\wt{G}_{2^t,3})$ is in degree $2^t-1$. Therefore, we will get one non-zero element in 
\bgd
\frac{F_{2^t-1}^{2^t-1}}{\textup{Image of }\underset{p+i=2^t-1,p\neq0,i\neq0}{\oplus}F_p^p\otimes F_i^i\textup{ in }F_{2^t-1}^{2^t-1}}.
\edd
As $H^{2^t-1}(W_{2,1}^{2^t})=\langle c_{2^t-1}\rangle$, it follows from Lemma \ref{important} that $c_{2^t-1}\in F_{2^t-1}^{2^t-1}$. Therefore, if for some $i\geq 0$ we have $\w_2^i c_{2^t-2}\in  F_{2i+2^t-2}^{2i+2^t-2}$ then $\w_2^{2^{t-1}-2}c_{2^t-2}c_{2^t-1}\in F_{3 \cdot 2^t-7}^{3 \cdot 2^t-7}=0$, which is a contradiction. Therefore, for any $0\leq i \leq 2^{t-1}-2$, 
\bgd
\w_2^i c_{2^t-2}\not\in F_{2i+2^t-2}^{2i+2^t-2}.
\edd
Similar arguments imply that $\w_2^i c_{2^t-2}c_{2^t-1}\not\in F_{2i+2^{t+1}-3}^{2i+2^{t+1}-3}$. Therefore, it is clear that 
\bgd
\frac{F_{r}^{r}}{\textup{Image of }\underset{p+i=r,p\neq0,i\neq0}{\oplus}F_p^p\otimes F_i^i\textup{ in }F_{r}^{r}}\neq 0
\edd
in degrees $r=2,3$ and $2^t-1$ only. Therefore, the indecomposables in $H^*(\wt{G}_{2^t,3})$ are $\wt{w}_2,~\wt{w}_3,~w_{2^t-1}.$\\[0.2cm]
For (b) the proofs for the three values of $n$ will be similar. We will prove it for $n=2^t-3$. According to \cite{Kor15}, the degree of the next indecomposable after $\wt{w}_2,~\wt{w}_3$ in $H^*(\wt{G}_{2^t-3,3})$ is $2^t-4.$ Observe that 
\bgd
E_\infty^{2^t-8,2}=\langle p_{2^t-8}\otimes \mathbbm{1}\rangle=\langle \bar{d}_{2^t-6}\rangle
\edd
where $\wt{w}_3 p_{2^t-8}=g_{2^t-5}$ and $p_{2^t-8}\neq 0$. We also observe $\wt{w}_2 p_{2^t-8}\neq0\in H^*(\wt{G}_{2^t-3,3})$. The ring structure obtained from the spectral sequence implies that $\w_2 d_{2^t-6}\not\in F_{2^t-4}^{2^t-4}.$ Therefore, exactly one of $d_{2^t-4}$ or $d_{2^t-4}+\w_2 d_{2^t-6}\in F_{2^t-4}^{2^t-4}$. \\
\hf If $d_{2^t-4}\in F_{2^t-4}^{2^t-4}$ then $\w_2^i d_{2^t-6}\not\in F_{2i+2^t-6}^{2i+2^t-6}$ for $i\leq \frac{2^t-6}{2}$      and $\w_2^i (d_{2^t-4}+\w_2 d_{2^t-6})\not\in F_{2i+2^t-4}^{2i+2^t-4}$ for $i<\frac{2^t-6}{2}$. We also have
\bgd
\w_2^{\frac{2^t-6}{2}}(d_{2^t-4}+\w_2 d_{2^t-6})=\w_2^{\frac{2^t-6}{2}}d_{2^t-4}.
\edd
\hf If $d_{2^t-4}+\w_2 d_{2^t-6}\in F_{2^t-4}^{2^t-4}$ then $\w_2^i d_{2^t-6}\not\in F_{2i+2^t-6}^{2i+2^t-6}$ for $i\leq \frac{2^t-6}{2}$. Hence, $\w_2^i d_{2^t-4}\not\in F_{2i+2^t-4}^{2i+2^t-4}$ for $i<\frac{2^t-6}{2}$. Again we have
\bgd
\w_2^{\frac{2^t-6}{2}}(d_{2^t-4}+\w_2 d_{2^t-6})=\w_2^{\frac{2^t-6}{2}}d_{2^t-4}.
\edd
 Therefore, in either cases, we observe that 
\bgd
\frac{F_{r}^{r}}{\textup{Image of }\underset{p+i=r,p\neq0,i\neq0}{\oplus}F_p^p\otimes F_i^i\textup{ in }F_{r}^{r}}=\left\{\begin{array}{rl}
0 & \textup{if $r>2^t-4$}\\
\Z_2 & \textup{if $r=2^t-4$}.
\end{array}\right.
\edd
Thus, by Lemma \ref{important}, $H^*(\wt{G}_{2^t-3,3})$ has only one indecomposable in degree $2^t-4$ and no indecomposable in higher degrees.\\
(c) We shall first do the case when $2^{t-1}<n\leq 2^t-4.$ In Proposition \ref{P0}, we have seen that the degree of the first indecomposable after $\wt{w}_2,\wt{w}_3$ is at least $n$. To obtain the indecomposables of degree at least $n$, we shall use Lemma \ref{important}. We shall first do this for an even integer $n$. The case when $n$ is odd will follow similarly.\\
\hf We see from the proof of Proposition \ref{P1}, $\w_2^i c_{n-2}\not\in F_{2i+n-2}^{2i+n-2}$ for $i\leq 2^{t-1}-\frac{n}{2}-2$ and $\w_2^i c_{n-1}\not\in F_{2i+n-1}^{2i+n-1}$ for $i\leq n-2^{t-1}-1$. We observe that $\w_2^i c_{n-2}c_{n-1}\not\in F_{2i+2n-3}^{2i+2n-3}$ as if $\w_2^i c_{n-2}c_{n-1}\in F_{2i+2n-3}^{2i+2n-3}$ then $\w_2^{\frac{n}{2}-2}c_{n-2}c_{n-1}\in F_{3n-7}^{3n-7}$, which is a contradiction. Therefore, via Proposition \ref{P1}, the only possible indecomposables will correspond to  
\bgd
\w_2^i c_{n-2}\,\,\,\textup{if}\,\,i\geq 2^{t-1}-\textstyle{\frac{n}{2}}-1\,\,\,\,\textup{and}\,\,\,\,\w_2^i c_{n-1}\,\,\,\textup{if}\,\,i\geq n-2^{t-1}.
\edd
If $i>2^{t-1}-\frac{n}{2}-1$ then
\bgd
\w_2^i c_{n-2}=\w_2^{i-(2^{t-1}-\frac{n}{2}-1)}\w_2^{2^{t-1}-\frac{n}{2}-1}c_{n-2}\in \underset{p+i=2i+n-2,p\neq0,i\neq0}{\oplus}F_p^p\otimes F_i^i.
\edd
Similarly, for $i>n-2^{t-1}$ we have 
\bgd
\w_2^i c_{n-1}\in\underset{p+i=2i+n-2,p\neq0,i\neq0}{\oplus}F_p^p\otimes F_i^i.
\edd
Thus, the only two possible indecomposables correspond to $\w_2^{2^{t-1}-\frac{n}{2}-1}c_{n-2}$ and $\w_2^{n-2^{t-1}}c_{n-1}.$\\
\hf We shall now prove that these two elements are non-zero in $\mathcal{P}_{2^t-4}$ and $\mathcal{P}_{3n-2^t-1}$ (cf. \eqref{P_r}). Let the degree of these two elements $\alpha_{i_1},\alpha_{i_2}$ be $i_1,i_2$ where $i_1<i_2.$ If $\alpha_{i_1}$ is in the image of $\underset{p+i=i_1,p\neq0,i\neq0}{\oplus}F_p^p\otimes F_i^i$ in $F_{i_1}^{i_1}$ then it is a polynomial in $\w_2$ in $H^*(W_{2,1}^n)$. As any monomial of degree at least $n$ in $\w_2$ is zero, the element $\alpha_{i_1}$ is indecomposable. Now, if $\alpha_{i_2}$ is in the image of $\underset{p+i=i_2,p\neq0,i\neq0}{\oplus}F_p^p\otimes F_i^i$, then we shall get relations of the form
\begin{eqnarray*}
\alpha_{i_2}=\w_2^{2^{t-1}-\frac{n}{2}-1}c_{n-2} & = & \w_2^{i_2-i_1}  \w_2^{n-2^{t-1}}c_{n-1}\,\,\textup{if $i_2=2^t-4>3n-2^t-1=i_1$} \\
\alpha_{i_2}=\w_2^{n-2^{t-1}}c_{n-1} & = & \w_2^{i_2-i_1} \w_2^{2^{t-1}-\frac{n}{2}-1}c_{n-2} \,\,\textup{if $i_1=2^t-4<3n-2^t-1=i_2$}. 
\end{eqnarray*}
In both cases, it leads to a contradiction as one side is odd while the other is even in degree.\\
\hf The case when $2^{t-1}<n\leq 2^t-4$ and $n$ is odd, will be proved similarly with necessary changes. 
\end{proof}
\hf\hf We do not have any natural choice for the indecomposables in degrees higher than $3$. Thus, if there is an indecomposable in degree $i$, we choose an arbitary indecomposable element in this degree and denote it by $w_{i}$. Note that $w_{2^t-4}w_{3n-2^t-1}=0$ as this element is in degree $3n-5$ which is beyond the dimension of $\wt{G}_{n,3}$.
\begin{defn}
Let $n$ be even and $n\neq 2^t,~2^t-2$. Define  
\bgd
v_{2^t-8}:=\wt{w}_2^{2^{t-1}-2-\frac{n}{2}}p_{n-4},\,\,v_{3n-2^t-5}:=\wt{w}_2^{n-2^{t-1}-1}p_{n-3}
\edd
where $\wt{w}_3 p_{n-4} = g_{n-1}$ and $\wt{w}_3 p_{n-3} = g_n +\wt{w}_2 g_{n-2}$ (cf. \eqref{pneven}). \\
\hf Let $n$ be odd and $n\neq 2^t-1,~2^t-3$. Define 
\bgd
v_{2^t-8}=\wt{w}_2^{2^{t-1}-2-\frac{n+1}{2}}p_{n-3},\,\,v_{3n-2^t-5}=\wt{w}_2^{n-2^{t-1}-1}q_{n-3}
\edd
where $\wt{w}_3 p_{n-3}=g_n$ and $\wt{w}_3 q_{n-3}=g_n+\wt{w}_2g_{n-2}$ (cf. \eqref{pnodd1} and \eqref{pnodd2}).
\end{defn}
From Proposition \ref{P1} and the discussion preceeding it, we conclude that  
\bge\label{w23van}
\wt{w}_2 v_{2^t-8}=0,\,\,\wt{w}_3 v_{2^t-8}=0,\,\,\wt{w}_2 v_{3n-2^t-5}=0,\,\,\wt{w}_3 v_{3n-2^t-5}=0.
\ede
It is crucial that $v_{2^t-8}$ and $v_{3n-2^t-5}$ are both non-zero elements in the cohomology ring. 
\begin{thm}\label{relations}
\textup{(a)} Let $2^{t-1}<n\leq 2^t-4$. Then we have the following relations:\\
\hf \textup{(i)} $v_{2^t-8}^2=0,\,v_{3n-2^t-5}^2=0,\,v_{2^t-8} v_{3n-2^t-5}=0$;\\
\hf \textup{(ii)} $v_{2^t-8} w_{2^t-4}=0,\,v_{3n-2^t-5} w_{3n-2^t-1}=0,\,w_{2^t-4}w_{3n-2^t-1}=0$;\\
\hf \textup{(iii)} $H^{3n-9}(\wt{G}_{n,3})=\langle v_{2^t-8} w_{3n-2^t-1}\rangle=\langle v_{3n-2^t-5} w_{2^t-4}\rangle$;\\
\hf \textup{(iv)} the following hold:\\[-0.4cm]
\begin{eqnarray*}
w_{2^t-4}^2=0,\,\,w_{3n-2^t-1}^2 & = & \mathcal{R}_1+w_{3n-2^t-1}\mathcal{R}_2+w_{2^t-4}\mathcal{R}_3\,\,\,\,\textup{if $3n-2^t-1<2^t-4$}\\
w_{3n-2^t-1}^2=0,\,\,w_{2^t-4}^2 & = & \mathcal{Q}_1+w_{3n-2^t-1}\mathcal{Q}_2+w_{2^t-4}\mathcal{Q}_3\,\,\,\,\textup{if $3n-2^t-1>2^t-4$}
\end{eqnarray*}
where $\mathcal{R}_i,\mathcal{Q}_i$'s are polynomials in $\wt{w}_2,\wt{w}_3$.\\
\hf \textup{(v)} For any $P(\wt{w}_2,\wt{w}_3)\neq0 \in H^*(\wt{G}_{n,3})$ we have either $P|v_{3n-2^t-5}$ or $P|v_{2^t-8}$, i.e., there exists a monomial $Q(\wt{w}_2,\wt{w}_3)\in \pi^*(H^*(G_{n,3}))$ such that $P\cdot Q=v_{3n-2^t-5}$ or $P\cdot Q=v_{2^t-8}$ respectively;\\
\hf \textup{(vi)} If $P(\wt{w}_2,\wt{w}_3)|v_{3n-2^t-5}$, then $P w_{2^t-4}\not\in \pi^*(H^*(G_{n,3}))$. Similarly, if $P|v_{2^t-8}$, then $P w_{3n-2^t-1}\not\in \pi^*(H^*(G_{n,3}))$.\\
\textup{(b)} If $n=2^t$ then 
\bgd
H^*(\wt{G}_{2^t,3})=\frac{\frac{\mathbb{Z}_2[\wt{w}_2,\wt{w}_3]}{\langle g_{2^t-2}, g_{2^t-1}\rangle}\otimes \mathbb{Z}_2[w_{2^t-1}]}{\langle w_{2^t-1}^2-P w_{2^t-1}\rangle}
\edd
for some $P(\wt{w}_2,\wt{w}_3)\in \pi^*(H^{2^t-1}(G_{2^t,3}))$.\\
\textup{(c)} If $n=2^t-1,~2^t-2,~2^t-3$ then 
\bgd
H^*(\wt{G}_{n,3})=\frac{\frac{\mathbb{Z}_2[\wt{w}_2,\wt{w}_3]}{\langle g_{n-2},g_{n-1},g_n\rangle}\otimes \mathbb{Z}_2[w_{2^t-4}]}{\langle w_{2^t-4}^2-P_1 w_{2^t-4}-P_2 \rangle}
\edd
for some $P_1(\wt{w}_2,\wt{w}_3),P_2(\wt{w}_2,\wt{w}_3)\in\pi^*(H^*(G_{n,3}))$ with $P_2=0$ for $n=2^t-2,~2^t-3$.
\end{thm}
\begin{proof}
(a) (i): It follows from the definition of $v_{2^t-8}$ and $\wt{w}_2 v_{2^t-8}=0$ that $v_{2^t-8}^2=0$ if $v_{2^t-8}$ has $\wt{w}_2$ as a factor, i.e., if $2^t-5>n$. When $n=2^t-4,2^t-5$ we get $v_{2^t-8}=p_{n-4}$ is a polynomial (of degree $n-4$) in $\wt{w}_2$ and $\wt{w}_3$. Due to \eqref{w23van}, it follows that $v_{2^t-8}^2=p_{n-4}^2=0$. Identical reasons imply that $v_{3n-2^t-5}^2=0$ as well as $v_{2^t-8} v_{3n-2^t-5}=0$. \\[0.2cm]
\hf (ii)-(iii): Case I : $3n-2^t-1<2^t-4$\\
Due to degree reasons (the degree of the classes being higher than the dimension $3n-9$ of $\wt{G}_{n,3}$) it follows that 
\bge\label{ww1}
w_{2^t-4} w_{3n-2^t-1}=0,\,\,w_{2^t-4}^2=0,\,\,v_{2^t-8} w_{2^t-4}=0.
\ede
By Poincar\'{e} duality, there exists $P+\lambda w_{3n-2^t-1}\in H^{3n-2^t-1}(\wt{G}_{n,3})$, where $P$ is a polynomial in $\wt{w}_2,\wt{w}_3$  and $\lambda\in\mathbb{Z}_2$, such that $v_{2^t-8}(P+\lambda w_{3n-2^t-1})=1\in H^{3n-9}(\wt{G}_{n,3}).$ The fact that any polynomial in $\wt{w}_2,\wt{w}_3$ of dimension $3n-9$ in $H^*(\wt{G}_{n,3})$ is $0$ implies that $\lambda=1$. Therefore, $v_{2^t-8}w_{3n-2^t-1}=1\in H^{3n-9}(\wt{G}_{n,3}).$ By Poincar\'{e} duality, there will be an element 
\bgd
\check{v}:=Q_1(\wt{w}_2,\wt{w}_3)+Q_2(\wt{w}_2,\wt{w}_3)w_{3n-2^t-1}+\lambda w_{2^t-4}\in H^{2^t-4}(\wt{G}_{n,3})
\edd
such that $\check{v}v_{3n-2^t-5}=1\in H^{3n-9}(\wt{G}_{n,3})$. Here $Q_1, Q_2$ are polynomials in $\wt{w}_2,\wt{w}_3$ and $\lambda\in\mathbb{Z}_2.$ Observe that there is no term $Q_3w_{3n-2^t-1}^2$ as $2(3n-2^t-1)>2^t-4$. Due to \eqref{w23van}, $Q_1,Q_2$ annihilates $v_{3n-2^t-5}$, whence $\lambda=1$ and $w_{2^t-4}v_{3n-2^t-5}=1$. This proves (iv). \\
\hf It remains to prove that $v_{3n-2^t-5} w_{3n-2^t-1}=0$. If not then there will be an element $u\in H^*(\wt{G}_{n,3})$ such that $v_{3n-2^t-5} w_{3n-2^t-1} u=1\in H^{3n-9}(\wt{G}_{n,3})$. From the previous paragraph, we conclude that $w_{3n-2^t-1} u=Q_1+Q_2 w_{3n-2^t-1}+w_{2^t-4}$. This is a contradiction as $w_{2^t-4}$ is an indecomposable in $H^*(\wt{G}_{n,3})$. Therefore,  $v_{3n-2^t-5} w_{3n-2^t-1}=0.$\\
\hf (ii)-(iii): Case II : $3n-2^t-1>2^t-4$\\
Due to degree reasons (the degree of the classes being higher than the dimension $3n-9$ of $\wt{G}_{n,3}$) it follows that 
\bge\label{ww2}
w_{2^t-4} w_{3n-2^t-1}=0,\,\,w_{3n-2^t-1}^2=0,\,\,v_{3n-2^t-5} w_{3n-2^t-1}=0.
\ede
By Poincar\'{e} duality, there exists $P+\lambda w_{2^t-4}\in H^{2^t-4}(\wt{G}_{n,3})$, where $P$ is a polynomial in $\wt{w}_2,\wt{w}_3$  and $\lambda\in\mathbb{Z}_2$, such that $v_{3n-2^t-5}(P+\lambda w_{2^t-4})=1\in H^{3n-9}(\wt{G}_{n,3}).$ The fact that any polynomial in $\wt{w}_2,\wt{w}_3$ of dimension $3n-9$ in $H^*(\wt{G}_{n,3})$ is $0$ implies that $\lambda=1$. Therefore, $v_{3n-2^t-5}w_{2^t-4}=1\in H^{3n-9}(\wt{G}_{n,3}).$ By Poincar\'{e} duality, there will be an element 
\bgd
\check{v}:=Q_1(\wt{w}_2,\wt{w}_3)+Q_2(\wt{w}_2,\wt{w}_3)w_{2^t-4}+\lambda w_{3n-2^t-1}\in H^{3n-2^t-1}(\wt{G}_{n,3})
\edd
such that $\check{v}v_{2^t-8}=1\in H^{3n-9}(\wt{G}_{n,3})$. Here $\lambda\in\mathbb{Z}_2$, $Q_1, Q_2$ are polynomials in $\wt{w}_2,\wt{w}_3$ and there is no term of the form $Q_3(\wt{w}_2,\wt{w}_3)w_{2^t-4}^2$ as $2(2^t-4)>3n-2^t-1$. Due to \eqref{w23van}, $Q_1,Q_2$ annihilates $v_{2^t-8}$, whence $\lambda=1$ and $w_{3n-2^t-1}v_{2^t-8}=1$. This proves (iv). \\
\hf It remains to prove that $v_{2^t-8} w_{2^t-4}=0$. Suppose that $v_{2^t-8}w_{2^t-4}\neq 0$. Then there will be an element $u\in H^*(\wt{G}_{n,3})$ such that $v_{2^t-8} w_{2^t-4} u=1\in H^{3n-9}(\wt{G}_{n,3})$. From the previous paragraph, we conclude that $w_{2^t-4} u=Q_1+Q_2 w_{2^t-4}+w_{3n-2^t-1}$. This is a contradiction as $w_{3n-2^t-1}$ is an indecomposable in $H^*(\wt{G}_{n,3})$. Therefore,  $v_{2^t-8} w_{2^t-4}=0.$\\

(iv)-(v): Consider the Serre spectral sequence corresponding to the fibre bundle $SO(3)\hookrightarrow V_3(\mathbb{R}^n)\xrightarrow{p_1} \wt{G}_{n,3}.$  Let the cohomology of $SO(3)\cong\mathbb{RP}^3$ be $\mathbb{Z}_2[a]/\langle a^4\rangle$. As $n\geq 9$ we note that (cf. \eqref{V3n}) 
\bgd
H^*(V_3(\mathbb{R}^n))\cong\Z_2 [a_{n-3},a_{n-2},a_{n-1}]/\langle a_{n-3}^2, a_{n-2}^2, a_{n-1}^2\rangle
\edd
contains no non-zero elements in cohomology of degree $2$ or $3$. Therefore, in the $E_2$-page, $d^2(1\otimes a)=\wt{w}_2$ and in the $E_3$-page $d^3(1\otimes a^2)=\wt{w}_3$. So, $v_{3n-2^t-5}\otimes a^3$ and $v_{2^t-8}\otimes a^3$ will survive to the $E_4$-page. \\[0.2cm]
Case I: $3n-2^t-1<2^t-4$\\
There is no non-zero cohomology in degree $3n-2^t-2$ for $V_3(\mathbb{R}^n)$. Due to \eqref{ww1}, the only non-zero elements in $E_4^{*,0}$ are $\bar{w}_{3n-2^t-1}^i~(i\geq0)$ and $\bar{w}_{2^t-4}$. Note that here $\bar{w}_{2^t-4}$ denotes its image in the $E_4$-page and is not to be confused with the notation used at the outset. Therefore, we conclude that $d^4(v_{3n-2^t-5}\otimes a^3)=\bar{w}_{3n-2^t-1}$ and  
\bgd
\bar{w}_{3n-2^t-1}^2=d^4((v_{3n-2^t-5}\otimes a^3)(w_{3n-2^t-1}\otimes 1))=d^4(v_{3n-2^t-5} w_{3n-2^t-1}\otimes a^3)=d^4(0)=0.
\edd
In particular, it follows that
\bge\label{wsq}
w_{3n-2^t-1}^2=\mathcal{R}_1+w_{3n-2^t-1}\mathcal{R}_2+w_{2^t-4}\mathcal{R}_3
\ede
where $\mathcal{R}_i, i=1,2,3$ are polynomials in $\wt{w}_2,\wt{w}_3$. It now follows that $E_4^{j,0}$ is non-zero only when $j=3n-2^t-1,~2^t-4$ and $E_4^{3n-2^t-1,0}=\langle \bar{w}_{3n-2^t-1}\rangle$, $E_4^{2^t-4,0}=\langle \bar{w}_{2^t-4}\rangle.$ We claim that $E_4^{j,3}$ is nonzero only for $j=3n-2^t-5, 2^t-8$ and in each case the dimension is $1.$ If possible, let $c_i\otimes a^3\in E_4^{i,3}$. There exists $c_j\in H^j(\wt{G}_{n,3})$ such that $i+j=3n-9$ and $c_i c_j=1.$ So, in the $E_4$- page $\overline{c_i\otimes a^3}\cdot \bar{c_j}\neq0.$ Therefore, for $\bar{c_j}$ to be non-zero we need $j=3n-2^t-1,2^t-4.$ So, $i=3n-2^t-5, 2^t-8.$  If $E_4^{3n-2^t-5,3}$ has rank at least $2$ then let $v_{3n-2^t-5},z$ be two linearly independent elements in it. As $z\in H^{3n-2^t-5}(\wt{G}_{n,3})$ is a non-zero element, by Poincar\'{e} duality and the fact that $2(3n-2^t-1)>2^t-4$ (as $n>2^{t-1}$) there exists polynomials $R_1,R_2$ (in $\wt{w}_2$ and $\wt{w}_3$) and $\lambda\in\Z_2$ such that
\bge\label{zdual}
z (R_1+R_2 w_{3n-2^t-1}+\lambda w_{2^t-4})=1\in H^{3n-9}(\wt{G}_{n,3}).
\ede 
Due to degree reasons, all the polynomials $R_i$'s, if they are non-zero, are of positive degree. If $z w_{2^t-4}=0$ then we may set $\lambda=0$. Otherwise, $z w_{2^t-4}=1$. Now $\lambda=0$ gives a contradiction when we consider \eqref{zdual} in $E_4$; the right hand side is non-zero while the co-factor of $z$ in the left hand side is zero. If $\lambda=1$ (and $z w_{2^t-4}=1$) then consider the Poincar\'{e} dual of $v_{3n-2^t-5}+z$ as in \eqref{zdual}. As $(v_{3n-2^t-5}+z)w_{2^t-4}=0$, the dual can be chosen such that $\lambda=0$ and we again arrive at a contradiction. Thus, the rank of $E^{3n-2^t-5,3}_4$ is one. Similar arguments apply to $E_4^{2^t-8,3}$, i.e., the only non-zero elements in $E_4^{*,3}$ are $v_{3n-2^t-5}\otimes a^3$ and $v_{2^t-8}\otimes a^3$. \\[0.2cm]
Case II: $3n-2^t-1>2^t-4$\\
There is no non-zero cohomology in degree $2^t-4$ for $V_3(\mathbb{R}^n)$. Due to \eqref{ww2}, the only non-zero elements in $E_4^{*,0}$ are $\bar{w}_{2^t-4}^i~(i\geq0)$ and $\bar{w}_{3n-2^t-1}$. Therefore, we conclude that $d^4(v_{2^t-8}\otimes a^3)=w_{2^t-4}$ and  
\bgd
\bar{w}_{2^t-4}^2=d^4((v_{2^t-8}\otimes a^3)(w_{2^t-4}\otimes 1))=d^4(v_{2^t-8} w_{2^t-4}\otimes a^3)=d^4(0)=0.
\edd
In particular, it follows that
\bge\label{wsq2}
w_{2^t-4}^2=\mathcal{Q}_1+w_{3n-2^t-1}\mathcal{Q}_2+w_{2^t-4}\mathcal{Q}_3
\ede
where $\mathcal{Q}_i, i=1,2,3$ are polynomials in $\wt{w}_2,\wt{w}_3$. It now follows that $E_4^{j,0}$ is non-zero only when $j=3n-2^t-1,~2^t-4$ and $E_4^{3n-2^t-1,0}=\langle \bar{w}_{3n-2^t-1}\rangle$, $E_4^{2^t-4,0}=\langle \bar{w}_{2^t-4}\rangle.$ We claim that $E_4^{j,3}$ is nonzero only for $j=3n-2^t-5, 2^t-8$ and in each case the dimension is $1.$ If possible, let $c_i\otimes a^3\in E_4^{i,3}$. There exists $c_j\in H^j(\wt{G}_{n,3})$ such that $i+j=3n-9$ and $c_i c_j=1.$ So, in the $E_4$- page $\overline{c_i\otimes a^3}\cdot \bar{c_j}\neq0.$ Therefore, for $\bar{c_j}$ to be non-zero we need $j=3n-2^t-1,2^t-4.$ So, $i=3n-2^t-5, 2^t-8.$  If $E_4^{3n-2^t-5,3}$ has rank at least $2$ then let $v_{3n-2^t-5},z$ be two linearly independent elements in it. As $z\in H^{3n-2^t-5}(\wt{G}_{n,3})$ is a non-zero element, by Poincar\'{e} duality and the fact that $2(2^t-4)>3n-2^t-1$ (as $n\leq 2^{t}-4$) there exists polynomials $R_1$ (in $\wt{w}_2$ and $\wt{w}_3$) and $\lambda\in\Z_2$ such that
\bge\label{zdual2}
z (R_1+\lambda w_{2^t-4})=1\in H^{3n-9}(\wt{G}_{n,3}).
\ede 
Due to degree reasons, the polynomials $R_1$, if non-zero, are of positive degree. If $z w_{2^t-4}=0$ then we may set $\lambda=0$. Otherwise, $z w_{2^t-4}=1$. Now $\lambda=0$ gives a contradiction when we consider \eqref{zdual2} in $E_4$; the right hand side is non-zero while the co-factor of $z$ in the left hand side is zero. If $\lambda=1$ (and $z w_{2^t-4}=1$) then consider the Poincar\'{e} dual of $v_{3n-2^t-5}+z$ as in \eqref{zdual2}. As $(v_{3n-2^t-5}+z)w_{2^t-4}=0$, the dual can be chosen such that $\lambda=0$ and we again arrive at a contradiction. Thus, the rank of $E^{3n-2^t-5,3}_4$ is one. Similar arguments apply to $E_4^{2^t-8,3}$, i.e., the only non-zero elements in $E_4^{*,3}$ are $v_{3n-2^t-5}\otimes a^3$ and $v_{2^t-8}\otimes a^3$. \\
\hf Now let us consider both cases (either $3n-2^t-1<2^t-4$ or $3n-2^t-1>2^t-4$) together. Let $P(\wt{w}_2,\wt{w}_3)\neq0\in H^*(\wt{G}_{n,3})$. If $P$ is either $v_{3n-2^t-5}$ or $v_{2^t-8}$, then we are done. Otherwise $P\otimes a^3\not\in E_4^{*,3}$ implies either $\wt{w}_2 P\neq0$ or $\wt{w}_3 P\neq 0.$ If $\wt{w}_2 P\neq0$, then $\wt{w}_2 P$ is either $v_{3n-2^t-5}$ or $v_{2^t-8}$ or $(\wt{w}_2 P)\otimes a^3 \not\in E_4^{\ast,3}$. Thus, either $\wt{w}_2 \wt{w}_2 P\neq 0$ or $\wt{w}_3 \wt{w}_2 P\neq 0$. After repeating this process finitely many times (as $H^*(\wt{G}_{n,3})$ is finite dimensional) we shall get $\wt{w}_2^r \wt{w}_3^s P$ is either $v_{3n-2^t-5}$ or $v_{2^t-8}$ for some $r,s$. Therefore, either $P|v_{3n-2^t-5}$ or $P|v_{2^t-8}.$ Finally note that \eqref{wsq} and \eqref{wsq2} imply (iv).\\[0.2cm]
(vi): Let $P|v_{3n-2^t-5}$ and suppose if possible that $P w_{2^t-4}\in \pi^*(H^*(G_{n,3}))$. Let $Q$ be a monomial in $\wt{w}_2,\wt{w}_3$ such that $PQ=v_{3n-2^t-5}$. As $Q\in \pi^\ast(H^\ast(G_{n,3}))$, it follows that 
\bgd
PQw_{2^t-4}=v_{3n-2^t-5} w_{2^t-4}=1\in\pi^*(H^{3n-9}(G_{n,3})),
\edd
which is a contradiction as $\pi^*(H^{3n-9}(G_{n,3}))=0$. \\

(b) From the proof of Theorem \ref{fifth}, we see that $\mathpzc{sp}^*:H^{2^t-1}(\wt{G}_{2^t,3})\to H^{2^t-1}(W_{2,1}^{2^t})$ that sends $w_{2^t-1}$ to $c_{2^t-1}.$ As $c_{2^t-1}^2=0$, it follows that $w_{2^t-1}^2=\wt{w}_3 P(\wt{w}_2,\wt{w}_3,w_{2^t-1})$ for some polynomial $P$. Observe that $E_\infty^{2^t-4,2}=\langle p_{2^t-4}\otimes \mathbbm{1}\rangle=\langle\bar{c}_{2^t-2}\rangle$ where $p_{2^t-4}\neq0\in H^{2^t-4}(\wt{G}_{2^t,3})$ satisfies $\wt{w}_3 p_{2^t-4}=g_{2^t-1}.$ So, from the proof of Theorem \ref{fifth} (a), we note that $\wt{w}_2^{2^{t-1}-2}p_{2^t-4}\neq0$ and 
\bgd
\wt{w}_2 \wt{w}_2^{2^{t-1}-2}p_{2^t-4}=\wt{w}_3 \wt{w}_2^{2^{t-1}-2}p_{2^t-4}=0\in H^*(\wt{G}_{2^t,3}).
\edd
Denote $\wt{w}_2^{2^{t-1}-2}p_{2^t-4}$ as $v_{2^{t+1}-8}.$ Also note that $v_{2^{t+1}-8}\in \pi^*(H^*(G_{2^t,3})).$\\
\hf Now, consider the spectral sequence corresponding to the fibre bundle $SO(3)\hookrightarrow V_3(\mathbb{R}^{2^t})\xrightarrow{p_1}\wt{G}_{2^t,3}.$ As observed earlier, $w_{2^t-1}^2=\wt{w}_3 P(\wt{w}_2,\wt{w}_3,w_{2^t-1})$ for some polynomial $P$. This implies that $\bar{w}_{2^t-1}^i=0$ in $E_4^{\ast,0}$ when $i\geq 2$. Thus, $E_4^{j,0}$ is nonzero only when $j=2^t-1$ and $E_4^{2^t-1,0}=\langle \bar{w}_{2^t-1}\rangle.$ So, $E_4^{i,3}$ is zero when $i\neq2^{t+1}-8$ and $E_4^{2^{t+1}-8}$ has dimension at most $1$. We already have noticed that $v_{2^{t+1}-8}\otimes a^3\in E_4^{2^{t+1}-8,3}.$ Thus, $v_{2^{t+1}-8}w_{2^t-1}=1\in H^{3(2^t-3)}(\wt{G}_{2^t,3})$ and $v_{2^{t+1}-8}$ is the element of highest degree in $\pi^*(H^*(G_{2^t,3})).$ As we did for (a), we again get if $Q(\wt{w}_2,\wt{w}_3)\neq0$ and $Q(\wt{w}_2,\wt{w}_3)\in\pi^*(H^*(G_{2^t,3}))$, then there is a monomial $Q_1(\wt{w}_2,\wt{w}_3)\in\pi^*(H^*(G_{2^t,3}))$ such that $Q(\wt{w}_2,\wt{w}_3)Q_1(\wt{w}_2,\wt{w}_3)=v_{2^{t+1}-8}$. This will again imply that $Q(\wt{w}_2,\wt{w}_3) w_{2^t-1}\not\in \pi^*(H^*(G_{2^t,3})).$ Therefore, we get the expected result. \\

(c) The proof is similar to the proof of (b).
\end{proof}

\begin{rem} \label{remark}
Results similar to Theorem \ref{relations} (a) (vi) hold in the following cases: \\
\textup{(a)} When $n=2^t$ there exists $v_{2^{t+1}-8}\neq0\in \pi^*(H^*(G_{n,3}))$ such that for any $P(\wt{w}_2,\wt{w}_3)\neq0\in \pi^*(H^*(G_{n,3}))$, we have a monomial $Q(\wt{w}_2,\wt{w}_3)\in \pi^*(H^*(G_{n,3}))$ with the property $P Q=v_{2^{t+1}-8}$ and $v_{2^{t+1}-8}w_{2^t-1}=1$.\\
\textup{(b)} When $n=2^t-1,~2^t-2,~2^t-3$ there exists $v_{3n-2^t-5}\neq0\in \pi^*(H^*(G_{n,3}))$ such that for any $P(\wt{w}_2,\wt{w}_3)\neq0\in \pi^*(H^*(G_{n,3}))$, we have a monomial $Q(\wt{w}_2,\wt{w}_3)\in \pi^*(H^*(G_{n,3}))$ with the property $PQ=v_{3n-2^t-5}$ and $v_{3n-2^{t}-5}w_{2^t-4}=1$
\end{rem}

\begin{cor} \label{cuplength}
For any integer $n\in[2^{t-1}+\frac{2^{t-1}}{3},2^t-2]$, the $\mathbb{Z}_2$-cup-length $\texttt{cup}(\wt{G}_{n,3})$ is bounded as follows:
\bgd
\max\{2^{t-1}-3,\textstyle{\frac{4(n-3)}{3}}-2^{t-1}+3\}\leq \texttt{cup}(\wt{G}_{n,3})\leq \frac{3(n-3)}{2}-2^{t-1}+3.
\edd
\end{cor}
\begin{proof}
The lower bound on $n$ implies that $3n-2^t-5> 2^t-8$. This further implies that
\bge\label{cuplb}
n-2^{t-1}-1+{\textstyle \frac{n-3}{3}}> 2^{t-1}-2-{\textstyle \frac{n}{2}+\frac{n-4}{3}}.
\ede
From Theorem \ref{relations}, \eqref{wsq2} and remark \ref{remark}, we observe that the monomial attaining $\texttt{cup}(\wt{G}_{n,3})$ will be of the form:\\
(1) $\wt{w}_2^{a_1}\wt{w}_3^{b_1}w_{2^t-4}$ or $\wt{w}_2^{a_2}\wt{w}_3^{b_2}w_{3n-2^t-1}$ if $n\in[2^{t-1}+\frac{2^{t-1}}{3},2^t-4]$;\\
(2) $\wt{w}_2^{a_1}\wt{w}_3^{b_1}w_{2^t-4}$ if $n=2^t-3,~2^t-2$.\\
It follows from \cite{Kor10} that $\wt{w}_2^{2^{t-1}-4}\neq 0$ in $H^{2^{t-1}-4}(\wt{G}_{2^{t-1},3})$. Note that 
\bgd
\wt{i}:\wt{G}_{n,3}\to \wt{G}_{n+1,3},\,\,\wt{i}^\ast:H^2(\wt{G}_{n+1,3})\xrightarrow{\cong}H^2(\wt{G}_{n,3})
\edd
induces an isomorphism that maps $\wt{w}_2$ to $\wt{w}_2$. Thus, 
\bgd
\wt{w}_2^{2^{t-1}-4}\neq 0\,\,\textup{in $H^{2^t-4}(\wt{G}_{n,3})$ for any integer $n\in [2^{t-1},2^t-1]$}.
\edd
Thus, due to Theorem \ref{relations} (v), we have
\bgd
\wt{w}_2^{2^{t-1}-4}=v_{2^t -8}\,\,\,\textup{or}\,\,\,\wt{w}_2^{2^{t-1}-4}\wt{w}_2^{a_1-2^{t-1}+4}\wt{w}_3^{b_1}=v_{3n-2^t-5}.
\edd
Thus, the cup-length is at least $2^{t-1}-3$. The description of $v_{3n-2^t-5}$ and $v_{2^t-8}$ (cf. Theorem \ref{relations} (iii)) implies that
\begin{eqnarray*}
n-2^{t-1}-1+\frac{n-3}{3} & \leq a_1+b_1\leq & \frac{3n-2^t-5}{2}\,\,\,\textup{for any $n\neq 2^t-3$}\\
n-2^{t-1}-1+\frac{n-2}{3} & \leq a_1+b_1\leq & \frac{3n-2^t-5}{2}\,\,\,\textup{for $n=2^t-3$}\\
2^{t-1}-2-\frac{n}{2}+\frac{n-4}{3} & \leq a_2+b_2\leq & 2^{t-1}-4\,\,\,\textup{for $n$ even}\\
2^{t-1}-2-\frac{n+1}{2}+\frac{n-3}{3} & \leq a_2+b_2\leq & 2^{t-1}-4\,\,\,\textup{for $n$ odd}.
\end{eqnarray*}
Therefore, combining the above with \eqref{cuplb}, we obtain
\bgd
\max\{2^{t-1}-3,n-2^{t-1}-1+\frac{n-3}{3}+1\}\leq \texttt{cup}(\wt{G}_{n,3})\leq \frac{3n-2^t-5}{2}+1
\edd
which simplifies to the bounds as claimed.
\end{proof}
Observe that the cup-length is attained by the length of a term of the form $\wt{w}_2^{a_1}\wt{w}_3^{b_1}w_{2^t-4}$. This is expected as among $w_{3n-2^t-5}$ and $w_{2^t-4}$, the latter is the indecomposable with the smaller degree in the above cases. 

\subsection{Characteristic Rank and Upper Characteristic Rank}\label{crucr}

In \cite{Kor15} it has been proved that $\texttt{charrank}(\wt{\gamma}_{n,k})\geq n-k+1$ for $k\geq 5$ and $n\geq 2k.$ Here we improve the previously known result.
\begin{thm} \label{k5}
If $k\geq 5$ and $n\geq 2k$, then characteristic rank of $\wt{\gamma}_{n,k}$ is at least $n-k+2$.
\end{thm}
To prove this theorem we first need to prove some basic results. Let $q$ be a non-negative integer and $q$ can be uniquely written as a finite sum $\sum_i a_i 2^i$ in the binary expansion, where $a_i=0$ or $1$. Let $S(q)$ be the set of non-negative integers $i$ for which $a_i=1$, written in decreasing order. For example, let us choose $23$. We can write $23$ as $23=2^4+2^2+2^1+2^0$. Then $S(23)=\{4,2,1,0\}$. So, $S(q)=\varnothing$ if and only if $q=0.$ If $\{p_1,p_2,p_3\}$ be a set of three non-negative integers with the property that $S(p_1)\cap S(p_2)=S(p_2)\cap S(p_3)=S(p_1)\cap S(p_3)=\varnothing$, then we say that {\it $\{p_1,p_2,p_3\}$ satisfies property $P$}. Now, we shall state the required fact (see \cite{Lu1878}). \\[0.2cm]
\textbf{Fact:} {\it If $p_1,p_2$ are two positive integers such $S(p_1)\cap S(p_2)=\varnothing$, then $p_1+p_2\choose p_1$ is odd.}\\[0.2cm]
The above fact and the following proposition will be used to prove Theorem \ref{k5}.
\begin{prpn} \label{3,4,5}
If $i\neq 1,2,7$, then $i$ can be written as $i=3p_1+4p_2+5p_3$ such that $\{p_1,p_2,p_3\}$ satisfies property $P$.
\end{prpn}
\begin{proof}
 First we prove the statement for $i=5.2^s+r$ for $r=1,2,7$ and for some non-negative integer $s$. We assume that $i=5.2^s+1$. Now we observe that $5.2^{s}+1-(5.2^{s-2}+1)=5.2^{s-2}.3$. Then putting $s=0$, we get $5.2^0+1=6$, which is divisible by 3. So, for an even integer $s$,  $5.2^s+1=3.\frac{5.2^s+1}{3}$. Here $\{p_1,p_2,p_3\}=\{\frac{5.2^s+1}{3},0,0\}$ which satisfies property $P$. When $s=1,$ $i=11=3.2+5.1$ where $\{2,0,1\}$ has property $P$. So, for any odd integer $s$, $5.2^s+1=3.(2+5.(2^{s-2}+\cdots+2))+5$ where $\{2+5.(2^{s-2}+\cdots+2),0,1\}$ has property $P$ as $2+5.(2^{s-2}+\cdots+2)$ is even.\\
\hf The other cases can be proved similarly using the facts that $5.2+2=3.4$, $5.2^2+2=3.6+4.1$, $5.2+7=3.4+5.1$, $5.2^0+7=3.4$. For $i=3,4,5,6,8,9,10$, this statement is obvious.
Now, we apply induction to prove the statement for any $i$ with $i>10.$ Fix an integer $i$ and assume that $j$ can be written as $3p_1+4p_2+5p_3$ where $\{p_1,p_2,p_3\}$ has the property $P$ for $j<i$ and $j\neq1,2,7$.  Let $s\geq1$ be an integer such that $5.2^s\leq i<5.2^{s+1}$. If $i-5.2^s=1,2,7$, then this case is already done. Otherwise, by induction $i-5.2^s$ can be written as $3p_1+4p_2+5p_3$ where $\{p_1,p_2,p_3\}$ has property $P$ and $p_3<2^s$.\\
\textbf{Case 1}:- If both $p_1,p_2<2^s$, then $i$
 can be written as $3p_1+4p_2+5(p_3+2^s)$. Then $S(p_3+2^s)=S(p_3)\cup \{s\}$. So, $S(p_1)\cap S(p_2)=S(p_2)\cap S(p_3+2^s)=S(p_1)\cap S(p_3+2^s)=\varnothing$. Hence, $\{p_1,p_2,p_3+2^s\}$ has the property $P$.\\
\textbf{Case 2}:- Assume $p_1\geq 2^s.$ Then $p_1<2^{s+1}$ as if $p_1\geq 2^{s+1}$, then $i-5.2^s-3.2^{s+1}\geq0,$ which gives a contradiction to $i<5.2^{s+1}.$ Similarly we can show that $p_2<2^s$, $p_3<2^s.$ Let $S(p_1)=\{s,s_1,\cdots\}$, $S(p_2)=\{t_1,\cdots\}$ and $S(p_3)=\{q_1,\cdots\}$. As $s_1,t_1,q_1<s$ we have
\begin{eqnarray*}
 i-5.2^s &=& 3(2^s+2^{s_1}+\cdots)+4(2^{t_1}+\cdots)+5(2^{q_1}+\cdots)\\
i &=& 3(2^{s_1}+\cdots)+4(2^{s+1}+2^{t_1}+\cdots)+5(2^{q_1}+\cdots)\\
i &=& 3(p_1-2^s)+4(p_2+2^{s+1})+5 p_3.
\end{eqnarray*}
We have $S(p_1-2^s)=S(p_1)-\{s\}$ and $S(p_2+2^{s+1})=S(p_2)\cup \{s+1\}$. Therefore, the statement is true for $i$.\\
\textbf{Case 3}:- Assume $p_2\geq2^s.$ Then, as we did in Case 2, we can show here too that $p_2<2^{s+1}$, $p_1<2^s$ and $p_3<2^s.$  Let $S(p_1)=\{s_1,\cdots\}$, $S(p_2)=\{s,t_1,\cdots\}$ and $S(p_3)=\{q_1,\cdots\}$. As $s_1,t_1,q_1<s$ we have 
\begin{eqnarray*}
i-5.2^s &=& 3(2^{s_1}+\cdots)+4(2^s+2^{t_1}+\cdots)+5(2^{q_1}+\cdots)\\
i &=& 3(2^{s+1}+2^s+2^{s_1}+\cdots)+4(2^{t_1}+\cdots)+5(2^{q_1}+\cdots)\\
i &=& 3(p_1+2^s+2^{s+1})+4(p_2-2^s)+5 p_3.
\end{eqnarray*}
We have $S(p_1+2^s+2^{s+1})=S(p_1)\cup\{s,s+1\}$ and $S(p_2-2^s)=S(p_2)-\{s\}.$ Therefore, the statement is true for $i.$
\end{proof}
Now we shall give a proof of Theorem \ref{k5}.
\begin{proof}
As $\bar{w}_i$ is the $i^\textup{th}$ degree term of the formal inverse of $(1+w_1+\cdots+w_k)$, we obtain from Proposition \ref{3,4,5}, $\bar{w}_i$ contains a monomial $w_3^{p_1}w_4^{p_2}w_5^{p_3}$ with coefficient ${p_1+p_2+p_3\choose p_1+p_2}\cdot {p_1+p_2\choose p_1}$ where $\{p_1,p_2,p_3\}$ satisfies property $P.$ Using the Fact we conclude that the coefficient ${p_1+p_2+p_3\choose p_1+p_2}\cdot {p_1+p_2\choose p_1}$ of $w_3^{p_1}w_4^{p_2}w_5^{p_3}$ is $1\in \Z_2$.
The Gysin sequence of the covering map $\mathbb{Z}_2\hookrightarrow \wt{G}_{n,k}\xrightarrow{\pi} G_{n,k}$ is the following:\\
$$\cdots\to H^{j-1}(G_{n,k})\xrightarrow{\cup w_1}H^j(G_{n,k})\xrightarrow{\pi^*}H^j(\wt{G}_{n,k})\to H^j(G_{n,k})\xrightarrow{\cup w_1}H^{j+1}(G_{n,k})\to\cdots.$$
By Theorem 2.1 of \cite{Kor15} and the discussion following its statement, it suffices to show that 
$$\textup{ker}(H^{n-k+2}(G_{n,k})\xrightarrow{\cup w_1}H^{n-k+3}(G_{n,k}))=0.$$
Let $p_{n-k+2}\in \textup{ker}(H^{n-k+2}(G_{n,k})\xrightarrow{\cup w_1}H^{n-k+3}(G_{n,k}))$, i.e.,
\bge\label{(1)}
p_{n-k+2}\cup w_1=a_1 w_1^2\bar{w}_{n-k+1}+a_2w_2\bar{w}_{n-k+1}+a_3w_1\bar{w}_{n-k+2}+a_4\bar{w}_{n-k+3}.
\ede
As $n-k+3>7$, $\bar{w}_{n-k+3}$ contains a monomial $w_3^{p_1}w_4^{p_2}w_5^{p_3}$ with non-zero coefficient. So, comparing the coefficients of $w_3^{p_1}w_4^{p_2}w_5^{p_3}$ in \eqref{(1)}, we get $a_4=0$. Now we get the reduced equation
\bgd
p_{n-k+2}\cup w_1=a_1w_1^2\bar{w}_{n-k+1}+a_2w_2\bar{w}_{n-k+1}+a_3w_1\bar{w}_{n-k+2}.
\edd
Reducing module $w_1$ and using Lemma \ref{Korbas} (ii), we obtain $p_{n-k+2}\cup w_1=a_1w_1^2\bar{w}_{n-k+1}+a_3w_1\bar{w}_{n-k+2}.$ Therefore, $p_{n-k+2}=0\in H^{n-k+2}(G_{n,k}).$ 
 \end{proof}
\noindent
Next we prove a result which gives a correspondence between the characteristic rank of the oriented tautological $k$-plane bundle $\wt{\gamma}_{n,k}$ over $\wt{G}_{n,k}$ and the upper characteristic rank of $\wt{G}_{n,k}$.  
\begin{thm} \label{uchar}
If $\texttt{charrank}(\wt{\gamma}_{n,k})=r-1$, where $r\neq2^t$ for some $t$, then $\texttt{ucharrank}(\wt{G}_{n,k})= r-1.$
\end{thm}
\begin{proof}
As $r-1$ is the characteristic rank of $\wt{\gamma}_{n,k}$, it follows that $H^i(G_{n,k})\xrightarrow{\pi^*} H^i(\wt{G}_{n,k})$ is surjective for $i< r$ and not surjective for $i=r$. We also know that $\texttt{ucharrank}(\wt{G}_{n,k})$ is at least $r-1$. Let $\xi$ be any vector bundle over $\wt{G}_{n,k}$. If we can show that $w_r(\xi)\in p^*(H^r(\wt{G}_{n,k}))$, then we can conclude that $\texttt{ucharrank}(\wt{G}_{n,k})$ is exactly $r-1.$ \\
\hf Let $r=2^ts$ where $s$ is an odd number with $s>1.$ Then by Wu's formula we have\\
$$Sq^{2^t}(w_{2^ts-2^t}(\xi))=\sum_{l=0}^{2^t}{{2^ts-2^t+l-2^t-1}\choose{l}}w_{2^t-l}(\xi) w_{2^t s - 2^t+l}(\xi)=A+{{2^ts-2^t-1}\choose{2^t}}w_{2^ts}(\xi),$$
where $A\in \pi^*(H^*(G_{n,k}))$. Now
$${{2^ts-2^t-1}\choose{2^t}}=\prod_{l=1}^{2^t}\frac{2^ts-2^t-l}{l}=\bigg(\prod_{l=1}^{2^{t-1}}\frac{2^ts-2^t-2l}{2l}\bigg)\cdot \bigg(\prod_{l=1}^{2^{t-1}}\frac{2^ts-2^t-2l+1}{2l-1}\bigg).$$
Let $l=2^{t^\prime}s^{\prime},$ where $s^\prime$ is odd. So, $2.2^{t^\prime}s^\prime\leq2^t,$ which implies $t^\prime+1\leq t.$ When $t^\prime+1=t$, we get $l=2^{t-1}.$ Then\\
$$\frac{2^ts-2^t-2l}{2l}=\frac{2^ts-2^t-2^{t^\prime+1}s^{\prime}}{2^{t^\prime+1}s^\prime}=\frac{2^{t-t^\prime-1}s-2^{t-t^\prime-1}-s^\prime}{s^\prime}.$$
When $t^\prime+1<t,$ then $\frac{2^ts-2^t-2l}{2l}$ is odd and when $t^\prime+1=t$, i.e., $l=2^{t-1}$, we have 
\bgd
\frac{2^ts-2^t-2l}{2l}=\frac{2^ts-2^t-2^t}{2^t}=s-2
\edd
is again an odd number. Therefore, $Sq^{2^t}(w_{2^ts-2^t}(\xi))=A+w_{2^ts}(\xi)$. Due to the naturality of $Sq^i,$ we have $Sq^{2^t}(w_{2^ts-2^t}(\xi))\in \pi^*(H^r(G_{n,k}))$. Therefore, $w_{2^ts}(\xi)\in \pi^*(H^r(G_{n,k}))$. As this is true for any vector bundle $\xi$, we infer that $\texttt{ucharrank}(\wt{G}_{n,k})=r-1.$
\end{proof}
We shall give one immediate implication of Theorem \ref{uchar}.
\begin{cor} \label{ucharrank for k=3}
The upper characteristic rank $\texttt{ucharrank}(\wt{G}_{n,3})=\texttt{charrank}(\wt{\gamma}_{n,3})$ for $n\geq8.$
\end{cor}
\begin{proof}
By \cite{Kor15} and \cite{PPR17} we have the full description of the characteristic rank of $\wt{G}_{n,3}.$ Consider the following observations:\\
(i) $2^r-4$ is not of the form $2^s$ when $r>3$, and\\
(ii) when $n>2^{t-1}$ and $3n-2^t-1<2^t-4$, then $2^{t-1}<3n-2^t-1<2^t-4,$ so $3n-2^t-1$ is also not of the form $2^s$.\\
The above two statements when combined together implies the result.
\end{proof}


\vspace*{0.4cm}

\noindent{\small D}{\scriptsize EPARTMENT OF }{\small M}{\scriptsize ATHEMATICS \& }{\small S}{\scriptsize TATISTICS, }{\small I}{\scriptsize NDIAN }{\small I}{\scriptsize NSTITUTE OF }{\small S}{\scriptsize CIENCE }{\small E}{\scriptsize DUCATION \& }{\small R}{\scriptsize ESEARCH, }{\small M}{\scriptsize OHANPUR, }{\small WB} {\footnotesize 741246, }{\small INDIA}\\
{\it E-mail address} : \texttt{somnath.basu@iiserkol.ac.in}\\[0.2cm]

\noindent{\small D}{\scriptsize EPARTMENT OF }{\small M}{\scriptsize ATHEMATICS \& }{\small S}{\scriptsize TATISTICS, }{\small I}{\scriptsize NDIAN }{\small I}{\scriptsize NSTITUTE OF }{\small S}{\scriptsize CIENCE }{\small E}{\scriptsize DUCATION \& }{\small R}{\scriptsize ESEARCH, }{\small M}{\scriptsize OHANPUR, }{\small WB} {\footnotesize 741246, }{\small INDIA}\\
{\it E-mail address} : \texttt{chakraborty.prateep@gmail.com}\\[0.2cm]


\begin{thebibliography}{10}

\bibitem{Bor53}
{\sc A.~Borel}, {\em Sur la cohomologie des espaces fibr\'es principaux et des
  espaces homog\`enes de groupes de {L}ie compacts}, Ann. of Math. (2), 57
  (1953), pp.~115--207.

\bibitem{Fuk08}
{\sc T.~Fukaya}, {\em Gr\"obner bases of oriented {G}rassmann manifolds},
  Homology Homotopy Appl., 10 (2008), pp.~195--209.

\bibitem{Kor10}
{\sc J.~Korba\v{s}}, {\em The cup-length of the oriented {G}rassmannians vs a
  new bound for zero-cobordant manifolds}, Bull. Belg. Math. Soc. Simon Stevin,
  17 (2010), pp.~69--81.

\bibitem{Kor15}
\leavevmode\vrule height 2pt depth -1.6pt width 23pt, {\em The characteristic
  rank and cup-length in oriented {G}rassmann manifolds}, Osaka J. Math., 52
  (2015), pp.~1163--1172.

\bibitem{KorRu16a}
{\sc J.~Korba\v{s} and T.~Rusin}, {\em A note on the {$\Bbb Z_2$}-cohomology
  algebra of oriented {G}rassmann manifolds}, Rend. Circ. Mat. Palermo (2), 65
  (2016), pp.~507--517.

\bibitem{KorRu16b}
\leavevmode\vrule height 2pt depth -1.6pt width 23pt, {\em On the cohomology of
  oriented {G}rassmann manifolds}, Homology Homotopy Appl., 18 (2016),
  pp.~71--84.

\bibitem{Lai74}
{\sc H.~F. Lai}, {\em On the topology of the even-dimensional complex
  quadrics}, Proc. Amer. Math. Soc., 46 (1974), pp.~419--425.

\bibitem{Lu1878}
{\sc E.~Lucas}, {\em Theorie des {F}onctions {N}umeriques {S}implement
  {P}eriodiques}, Amer. J. Math., 1 (1878), pp.~289--321.

\bibitem{MiSt}
{\sc J.~W. Milnor and J.~D. Stasheff}, {\em Characteristic classes}, Princeton
  University Press, Princeton, N. J., 1974.
\newblock Annals of Mathematics Studies, No. 76.

\bibitem{NaTh14}
{\sc A.~C. Naolekar and A.~S. Thakur}, {\em Note on the characteristic rank of
  vector bundles}, Math. Slovaca, 64 (2014), pp.~1525--1540.

\bibitem{PPR17}
{\sc Z.~Z. Petrovi\'c, B.~I. Prvulovi\'c, and M.~Radovanovi\'c}, {\em
  Characteristic rank of canonical vector bundles over oriented {G}rassmann
  manifolds {$G_{3,n}$}}, Topology Appl., 230 (2017), pp.~114--121.

\bibitem{Rus17}
{\sc T.~Rusin}, {\em A note on the characteristic rank of oriented {G}rassmann
  manifolds}, Topology Appl., 216 (2017), pp.~48--58.

\end{thebibliography}
\end{document}